\documentclass[12pt]{article} 

\usepackage{geometry} 
\usepackage{graphicx,float}	
\graphicspath{{figures/}}	
\usepackage{pgfplots}	
\pgfplotsset{compat=newest} 
\pgfplotsset{plot coordinates/math parser=false} 
\newlength\figureheight 
\newlength\figurewidth 
	
\usepackage[]{amsmath,amssymb,amsthm}
\usepackage{todonotes}
\usepackage{hyperref}
 
\newcommand{\N}{\mathbb{N}}
\newcommand{\R}{\mathbb{R}}
\newcommand{\Z}{\mathbb{Z}}

\newcommand{\cC}{\mathcal{C}}
\newcommand{\cO}{\mathcal{O}}

\newcommand{\hd}{\hat d}

\newcommand{\tX}{\tilde X}
\newcommand{\tY}{\tilde Y}

\DeclareMathOperator*{\Inv}{Inv}
\DeclareMathOperator*{\argmin}{argmin}

\newtheorem{proposition}{Proposition}
\theoremstyle{definition}
\newtheorem{experiment}{Experiment}

\title{On the sighting of unicorns: a variational approach
to computing invariant sets in dynamical systems}
\author{
Oliver Junge
\footnote{Center for Mathematics, Technical University of Munich, 85747 Garching, Germany, \texttt{oj@tum.de}. Supported by the DFG Collaborative Research Center SFB/TR 109 ``Discretization in Geometry and Dynamics''.}
\and
Ioannis G.~Kevrekidis
\footnote{Department of Chemical and Biological Engineering and PACM, Princeton University, Princeton, New Jersey 08544, USA, also Institute of Advanced Studies, TU Munich, and ZIB, FU Berlin, \texttt{yannis@princeton.edu}. The work of I.G.K is also partially supported by the US National Science Foundation (CBET and CDS\&E).}
}

\date{\today}

\begin{document}

\maketitle

\begin{abstract}
We propose to compute approximations to general invariant sets in dynamical systems by minimizing the distance between an appropriately selected finite set of points and its image under the dynamics.  We demonstrate, through computational experiments that this approach can successfully converge to approximations of (maximal) invariant sets of arbitrary topology, dimension and stability as, e.g., saddle type invariant sets with complicated dynamics.  We further propose to extend this approach by adding a Lennard-Jones type potential term to the objective function which yields more evenly distributed approximating finite point sets and perform corresponding numerical experiments.
\end{abstract}

\section{Introduction}

One central question in dynamical systems theory is to understand the existence and structure of \emph{invariant sets}.  Basic and important examples for invariant sets are fixed points/equi\-libria, periodic and quasiperiodic orbits and their associated stable and unstable manifolds.  In systems with chaotic behaviour, invariant sets with complicated topology may exist.  A plethora of numerical techniques has been developed in order to approximate invariant sets computationally: Straightforward \emph{simulations} (or more generally \emph{indirect methods}) typically reveal parts of some invariant set, e.g.\ some \emph{attractor} of the system, cf.\ e.g.\ \cite{Stuart:1996ve}. \emph{Direct methods} focus on invariant sets of some particular type or topology like the examples mentioned above. While indirect methods are restricted to invariant sets which are (asymptotically) stable in forward or backward time, direct methods can compute invariant sets of saddle type.  However, they include knowledge about the structure of the invariant set into the design of the method, in particular on how to properly parametrize the set, cf.\ e.g.\ \cite{Kevrekidis:1985ev,Beyn:90,KrauskopfOsinga:97,Doedel:al5wFh0U,Cliffe:2000dh,Beyn:2002ie}. In contrast, \emph{set oriented} techniques are capable of approximating invariant set without any\emph{a priori} knowledge on its structure \cite{DeHo97a,Dellnitz:1996ts,DeFrJu01a,DeJu02a}.  In these, the set under consideration is covered by a subset of a cubical decomposition of phase space.  While these box coverings provide a rigorous outer approximation to some invariant set, they do not provide a parametrization which varies smoothly in case that the invariant set varies smoothly with some system parameter.

The approach described in this paper is motivated by the desire to compute approximations to invariant sets of arbitrary topology, dimension and stability type which do vary smoothly as mentioned.    We propose to approximate some invariant set by a finite scattered point cloud which minimizes a certain objective functional (cf.\ \cite{BOLLT:2005cs} for another variational approach based on the lifetime of trajectories).  In its most basic form, this functional is simply the distance (given by some metric on sets, as e.g.\ the Hausdorff metric) between the point cloud and its image under the dynamics.  We give computational evidence that already this basic approach yields useful approximations, if the invariant set is (sufficiently strongly) hyperbolic.  We further propose to augment this basic functional by a second term which penalizes a ``too uneven'' distribution of the point cloud.  Here, we use a Lennard-Jones potential for this purpose.  Our numerical experiments suggest that this indeed improves the approximation quality if the involved parameters are chosen appropriately.

\section{Invariant sets}
\label{sec:invariant}

We consider a discrete-time dynamical system
\[
x_{k+1}=f(x_k), \quad k=0,1,2,\ldots,
\]
where $f:\R^d\to \R^d$ is a diffeomorphism (e.g.\ an explicit mapping or the time-$T$-map of some ordinary differential equation).  A set $X\subset \R^d$ is \emph{invariant} if
\[
X=f(X).
\] 
Simple examples for invariant sets are fixed points $\bar x = f(\bar x)$ or periodic orbits $X=\{x_0,\ldots,x_{p-1}\}$, $x_{k+1\mod p}=f(x_k)$.  If $X\subset \R^d$ is invariant then, by continuity of $f$, its closure is invariant as well and so in the following we can restrict our considerations to closed invariant sets.  In fact, we will be concerned with compact invariants sets only: Given some compact set $Q\subset \R^d$,  the set
\[
\Inv(Q) = \{x\in Q \mid f^k(x)\in Q \text{ for all } k\in\Z\}.
\]
is the \emph{maximal invariant set} within $Q$.  By definition, it contains all invariant sets which are contained in $Q$.  In many cases, e.g.\ in the numerical experiments below, $\Inv(Q)$ is independent of $Q$ if $Q$ is chosen large enough.  

\section{A variational scheme for invariant sets}
\label{sec:scheme}

Our approach to computing compact invariant sets will be based on minimizing the distance between some compact set $X\subset\R^d$ and its image $f(X)\subset\R^d$.  Let $\cC$ be the set of non-empty compact subsets of $\R^d$ and let $d:\cC\times\cC\to [0,\infty)$ be a metric on $\cC$.  Then, 
\begin{equation}\label{eq:metric}
X=f(X) \quad \text{ if and only if } \quad d(X,f(X)) = 0.
\end{equation}
In any numerical computation, we can only work on some subset of $\cC$ which can be described by finitely many parameters. On this subset, we cannot expect to satisfy $d(X,f(X))=0$.  The idea of our approach is to \emph{minimize} the (``energy'') functional $E:\cC\to [0,\infty)$, 
\begin{equation}\label{eq:vp}
E(X)=d(X,f(X)),
\end{equation}
on some suitable subset $\tilde\cC\subset\cC$ instead. 

Let $B_r(0)\subset\R^d$ be the ball centered at $0$ with radius $r$ and recall that the subset relation $\subseteq$ is a partial order on $\cC$.  
\begin{proposition}
Suppose that $\Inv(B_r(0))=S$ for some $S\in\cC$ for all sufficiently large $r$.  Then the set $S$ is the unique minimizer of $E$ which is maximal w.r.t.\ the subset relation.
\end{proposition} 

\begin{proof}
By (\ref{eq:metric}) and the definition of $E$, any minimizer of $E$ is an invariant set.  Thus, the union $U=\cup_{X\in\cC, X=f(X)} X$ of all compact invariant sets is a minimizer.  Further, since it contains all other minimizers from $\cC$, it is the unique set which is maximal w.r.t.\ the subset relation.
\end{proof}	 

\paragraph{The Hausdorff metric.}  A common way to measure distances between compact sets is via the Hausdorff metric which is defined as follows: For any non-empty set $X\subset \R^d$, the distance of a point $y\in\R^d$ from $X$ is
\[
d(y,X) = \inf_{x\in X} \|y-x\|_2.
\]
The distance of a second non-empty set $Y\subset\R^d$ from $X$ is
\[
d(Y,X) = \sup_{y\in Y} d(y,X)
\]
and since this distance is not symmetric one defines the \emph{Hausdorff metric}
\begin{align*}
d_H(X,Y) & =\max\left\{d(Y,X),d(X,Y)\right\}\\
& = \max\left\{\sup_{y\in Y} \inf_{x\in X} \|y-x\|_2,\sup_{x\in X} \inf_{y\in Y} \|y-x\|_2\right\}
\end{align*}
between $X$ and $Y$.  Note that $(\cC,d_H)$ is complete.

\paragraph{A modified Hausdorff metric.}

As mentioned, we are going to minimize the energy functional (\ref{eq:vp}) on some subset of $\cC$. In fact, we will simply use finite subsets $\tX=\{x_1,\ldots,x_n\}\subset\R^d$ (i.e.\ \emph{point clouds}) for this purpose, such that $E$ can be seen as a function on $\R^{nd}$, where $n$ is the (fixed) number of points in these subsets.  Unfortunately, $E:\R^{nd}\to [0,\infty)$ is not smooth and this prevents us from using standard schemes for the minimization. We therefore employ the following modified Hausdorff distance instead: We use
\[
\hd(y,\tX) = d(y,\tX)^2 = \min_{x\in \tX} \|y-x\|^2_2
\]
in order to measure the distance of some point $y\in\R^d$ from some non-empty finite set $\tX$. We further define the distance
\[
\hd(\tY,\tX) = \frac{1}{|\tY|}\sum_{y\in \tY} \hd(y,\tX) = \frac{1}{|\tY|}\sum_{y\in \tY} \min_{x\in \tX} \|y-x\|^2_2  
\]
of some non-empty finite set $\tY$ from $\tX$ ($|\tilde Y|$ denotes the number of points in $\tilde Y$) and finally define the Hausdorff like distance
\[
\hd_H(\tX,\tY) = \frac12(\hd(\tX,\tY) + \hd(\tY,\tX))
\]
between two non-empty finite sets $\tX$ and $\tY$.
Note that $\hd_H$ is a metric on the set of non-empty finite subsets of $\R^d$.  For some set $\tX=\{x_1,\ldots,x_n\}\subset\R^d$, the corresponding energy functional reads explicitly
\begin{align}\label{eq:vp_constr}
\nonumber 
\hat E(x_1,\ldots,x_n) & = \hat d_H(\tX,f(\tX))\\
& = \frac{1}{2n} \sum_{i=1}^n \min_{j=1:n} \|x_i-f(x_j)\|^2_2 + \frac{1}{2n} \sum_{i=1}^n \min_{j=1:n} \|f(x_i)-x_j\|^2_2\\
\nonumber & = \frac{1}{2n} \sum_{i=1}^n \|x_i-f(x_{j(i)})\|^2_2 + \frac{1}{2n} \sum_{i=1}^n \|f(x_i)- x_{j(i)}\|^2_2,
\end{align}
where $j(i)=\argmin_{j=1,\ldots,n} \|x_{j}-f(x_i)\|^2$.
 
\paragraph{Implementation.} We are going to minimize the energy functional $\hat E$ by a standard Quasi-Newton scheme, namely the limited memory BFGS scheme as implemented in the Matlab function \texttt{fminlbfgs}\footnote{by Dirk-Jan Kroon, University of Twente}.  In order to compute the distance $\hat d(y,X)$ of some point $y$ from some finite set $\tX$, we employ a kd-tree based search for some point $x=x(y)\in \tX$ which is closest to $y$.  This is conveniently implemented in the \texttt{knnsearch} command in Matlab.  In fact, \texttt{knnsearch} can return the $k\in\N$ nearest neighbours at once and each query of this type takes $\cO(\log |\tX|)$ time.  Overall, this translates into a running time of $\cO(|\tX|\log(\tX))$ and all the examples in the following section only take a few seconds to run on a recent machine.  For $|\tX|=10^4$, the runtime will be a few minutes.

\section{Computational experiments}
\label{sec:experiments}

\begin{experiment}[Fixed point in 1d]
Let us start by the simplest possible example: A linear map on the line: We consider $f:\R\to\R$, $f(x)=a x$ with $a=0.1$ and $a=10$. The maximal invariant set in $Q=[-1,1]$ is $\Inv(Q)=\{0\}$.   We initialize $\tX$ with 40 points, chosen randomly from $[-1,1]$ according to a uniform distribution and terminate the BFGS iteration  as soon as $\|\nabla\hat E\|_\infty < 10^{-6}$. Figure~\ref{fig:linear1d} shows the evolution of $\tX$ in course of the optimization for both values of $a$.  The BFGS iteration terminates after 21 resp.\ 18 steps with an $\hat E$ value of around $10^{-11}$, the Hausdorff distance of $\tX$ from $\{0\}$ is $\approx 10^{-6}$ for $a=0.1$ and $\approx 2\cdot 10^{-5}$ for $a=10$.
  
\begin{figure}[H]
\begin{center}
\includegraphics[trim = 0in 3in 0.5in 3in, clip,width=0.49\textwidth]{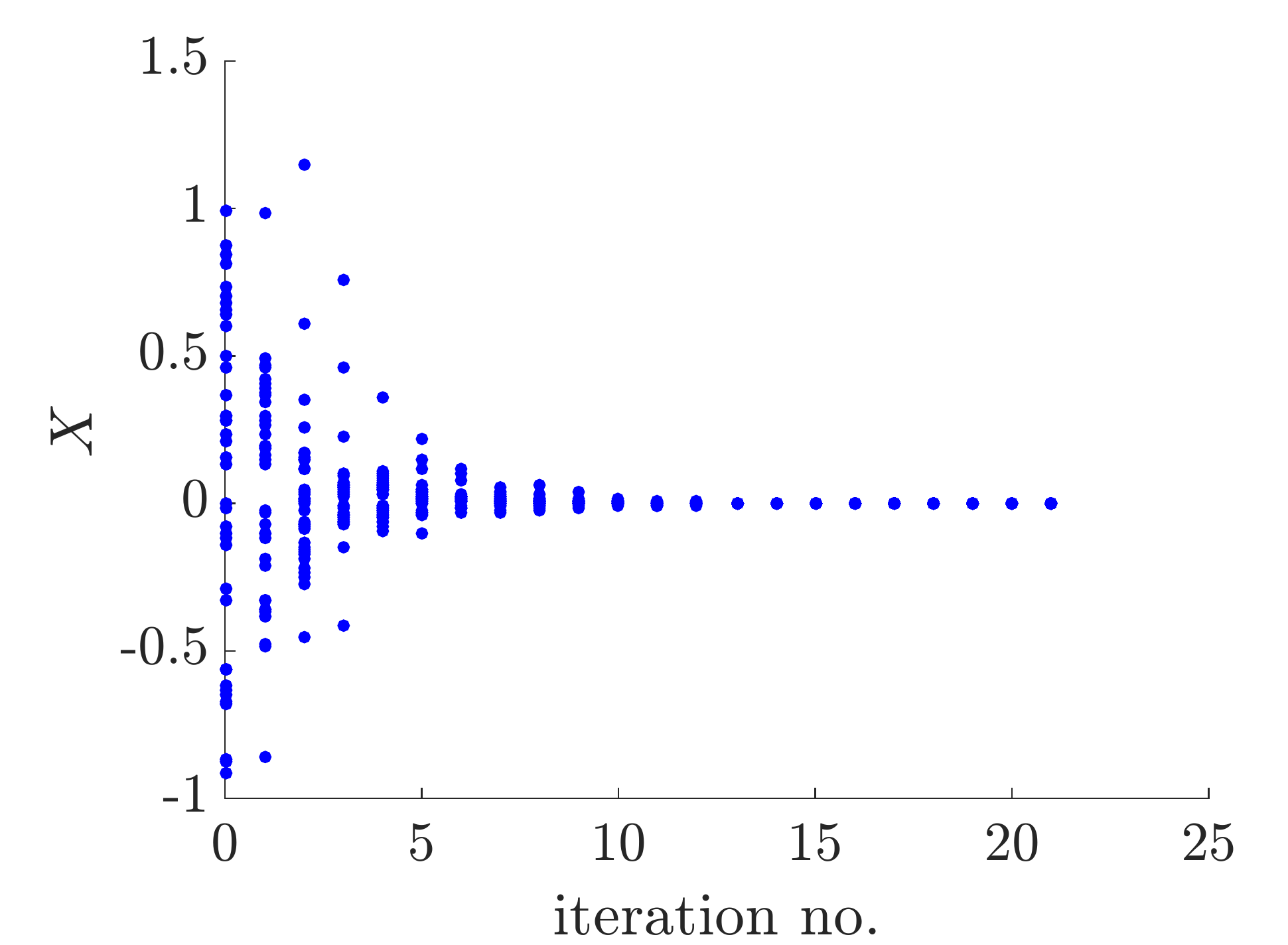}
\includegraphics[trim = 0in 3in 0.5in 3in, clip,width=0.49\textwidth]{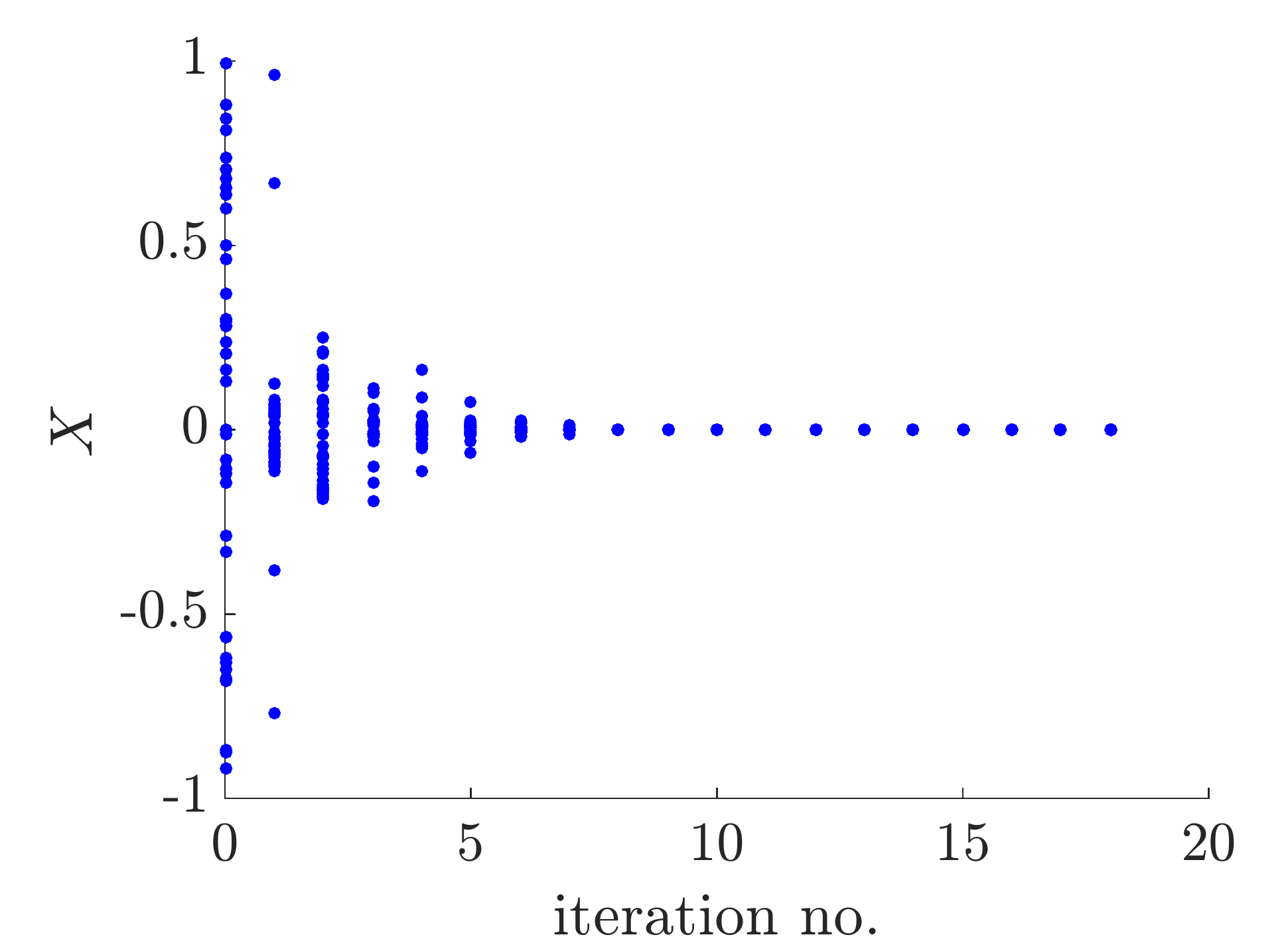}
\caption{Linear map on the line, evolution of $\tX$ in course of the optimization. Left: $a=0.1$, right: $a=10$.}
\label{fig:linear1d}
\end{center}
\end{figure}

The speed of convergence seems to strongly depend on the contraction constant $a$: Figure~\ref{fig:linear1d_a_close_to_1} shows the evolution of $\tX$ in course of the BFGS iteration for $a=1.1$ (left) and $a=1.01$ (right).  While in both cases the objective function value is less than $10^{-8}$, the Hausdorff distance of $\tX$ from $\{0\}$ is still rather large, namely $\approx 0.003$ for $a=1.1$ and $\approx 0.1$ for $a=1.01$, even after a much larger number of iterations.

\begin{figure}[H]
\begin{center}
\includegraphics[trim = 0in 3in 0.5in 3in, clip,width=0.49\textwidth]{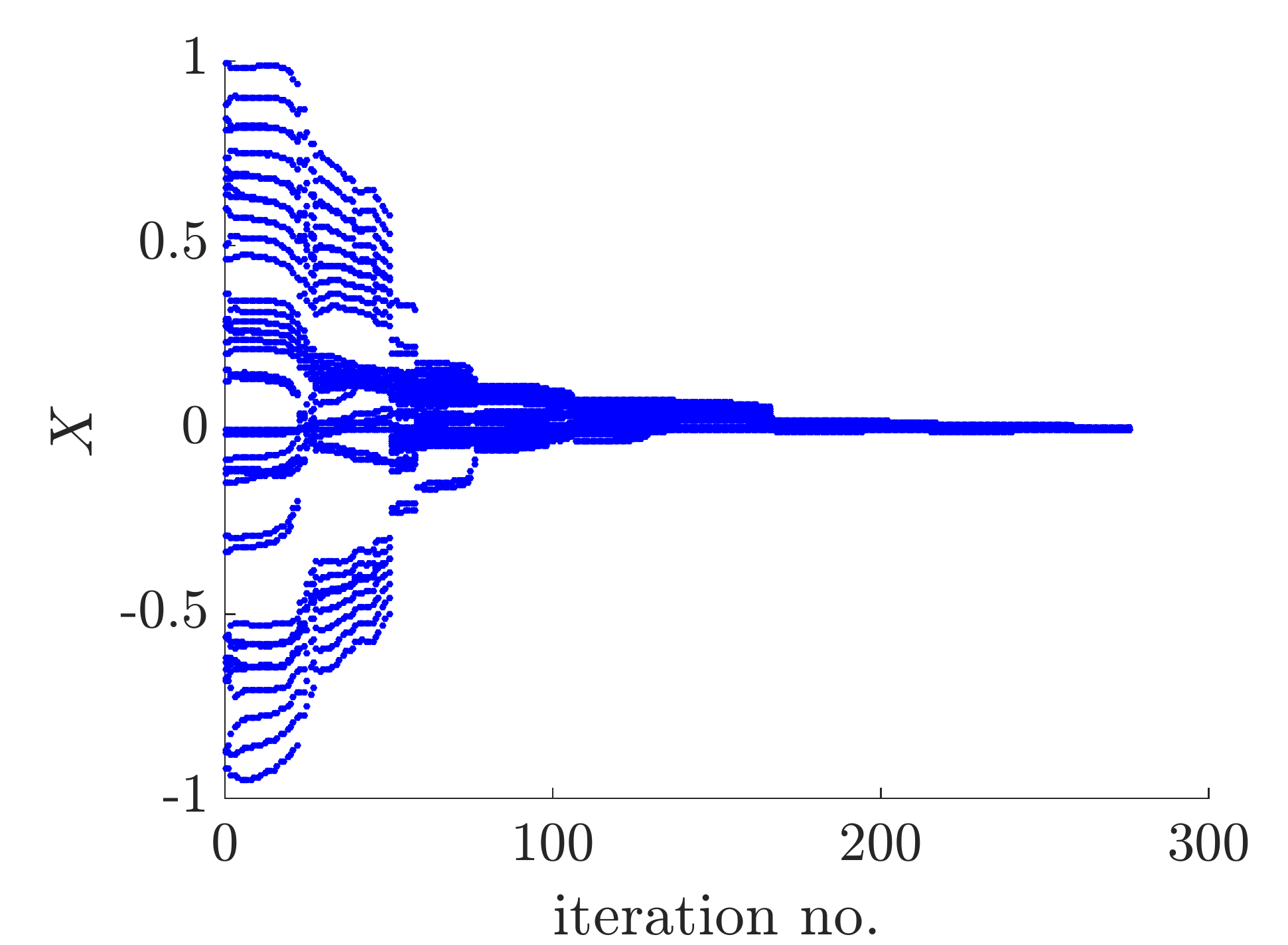}
\includegraphics[trim = 0in 3in 0.5in 3in, clip,width=0.49\textwidth]{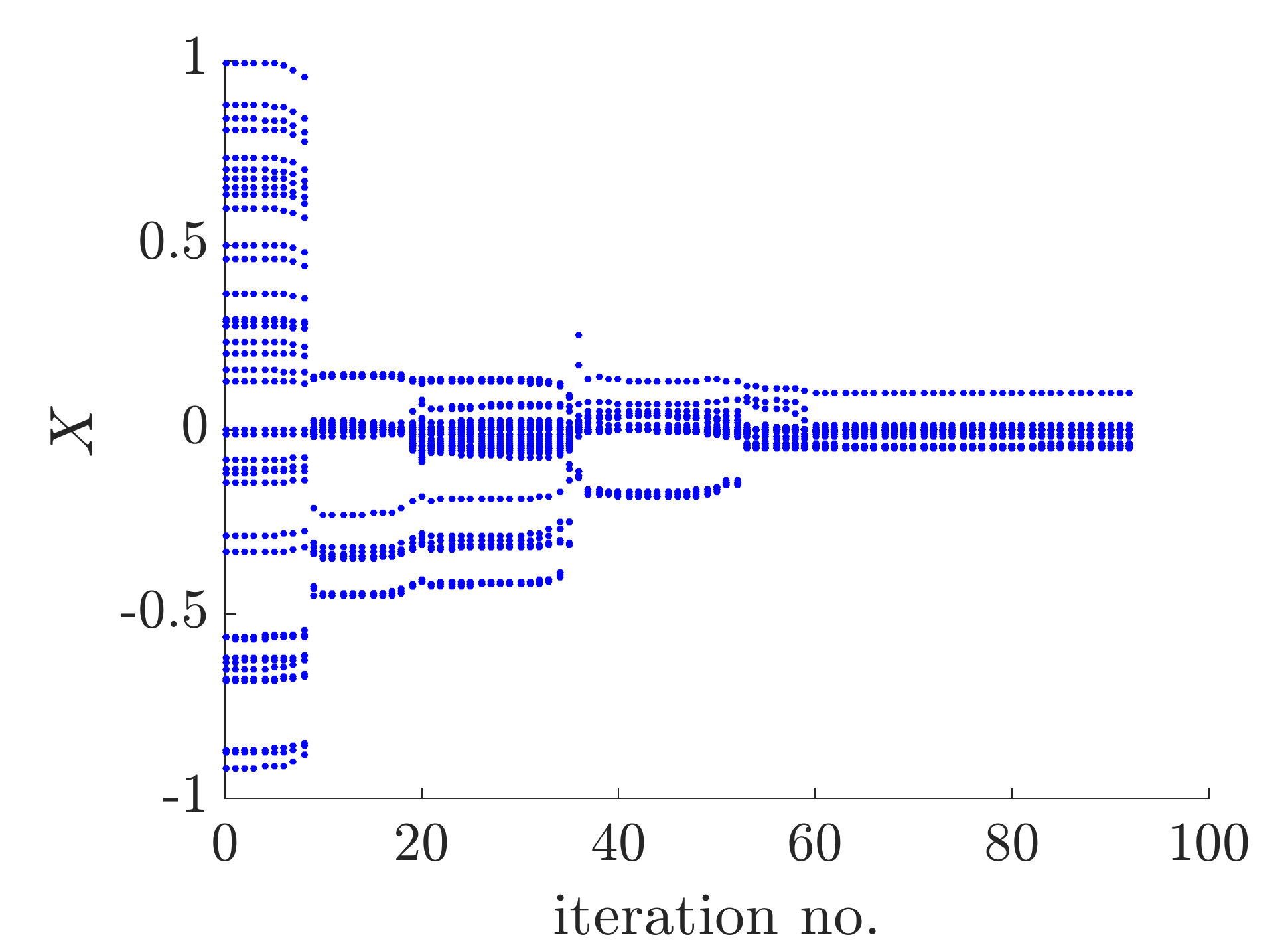}
\caption{Linear map on the line, evolution of $\tX$ in course of the optimization. Left: $a=1.1$, right: $a=1.01$.}
\label{fig:linear1d_a_close_to_1}
\end{center}
\end{figure}
\end{experiment}

\begin{experiment}[A connecting orbit in 1d]
We next consider a nonlinear map on the line for which the maximal invariant set is the interval $[0,1]$, the map is $f(x) = x + a x(1-x)$ with $a=0.8$. It possesses two fixed points, namely $\bar x_1=0$ (unstable) and $\bar x_2=1$ (stable).  Points within $(0,1)$ are heteroclinic to these two equilibria so that the maximal invariant set within any set $Q$ covering $[0,1]$ is the interval $[0,1]$.  We choose $Q=[-1,2]$, initialize $X$ by a set of points chosen randomly from $[-1,2]$ according to a uniform distribution.  Figure~\ref{fig:connecting_1d} shows the evolution of $\tX$ in course of the BFGS iteration for $n=100$ (left), as well as the Hausdorff distance $d_H$ between $\tilde X$ and $[0,1]$ (approximated by computing $d_H$ between $\tilde X$ and a grid of $10^4$ points in $[0,1]$).

\begin{figure}[H]
\begin{center}
\includegraphics[trim = 0in 3in 0.5in 3in, clip,width=0.49\textwidth]{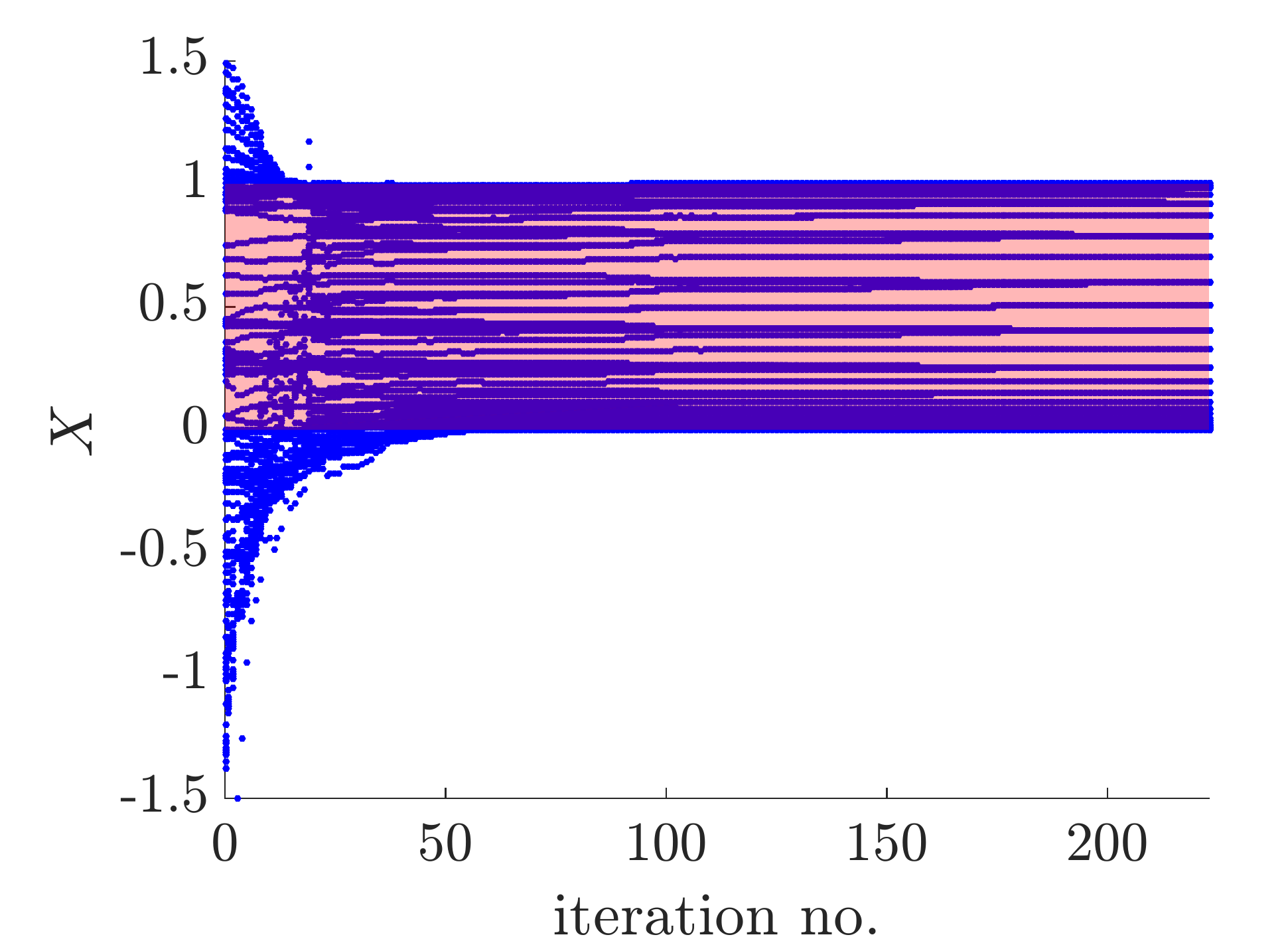}
\includegraphics[trim = 0in 3in 0.5in 3in, clip,width=0.49\textwidth]{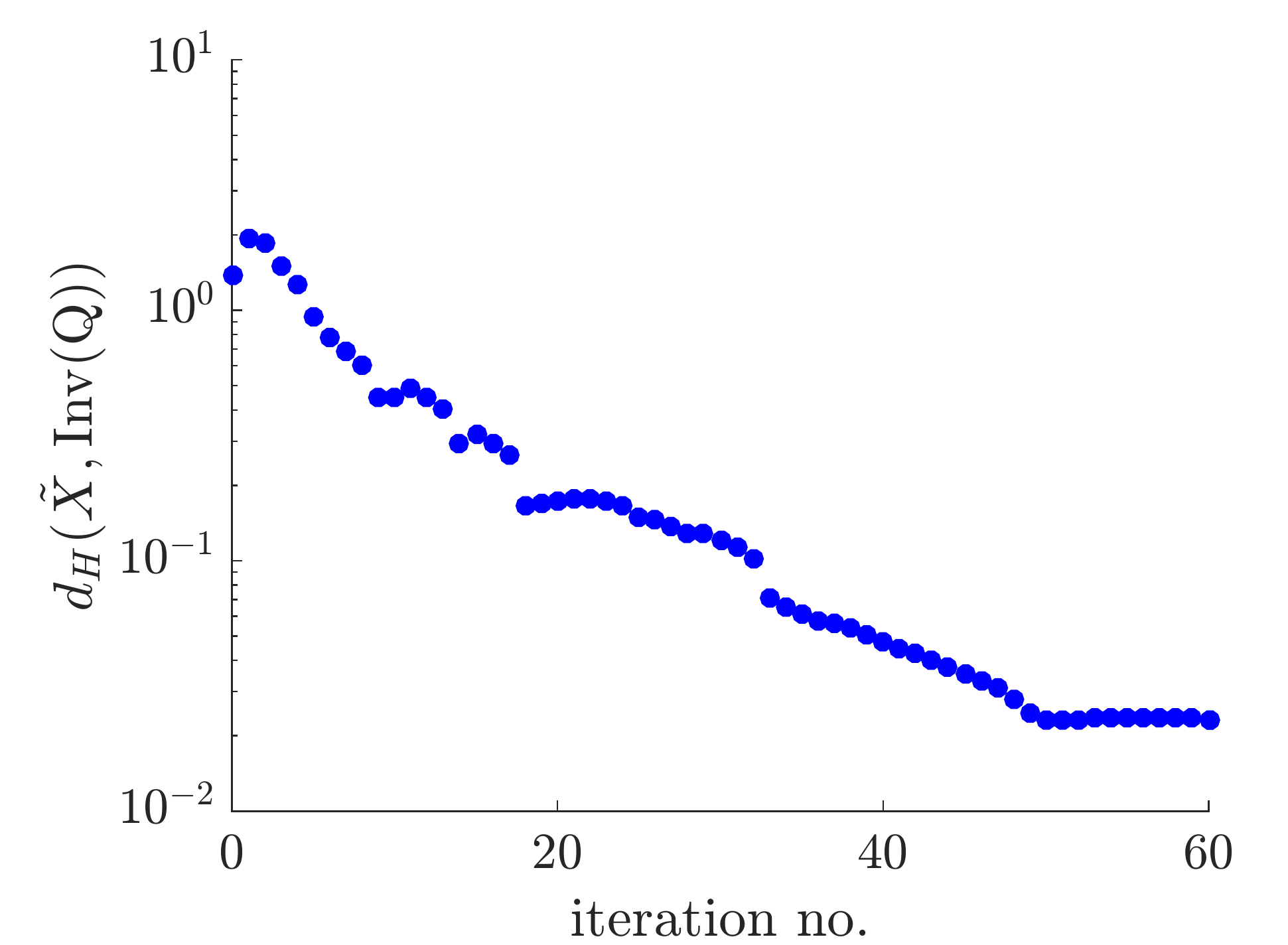}
\caption{Connecting orbit on the line: approximation of $\Inv(Q)=[0,1]$ (lightly red shaded) by a set $\tX$ of $n=100$ (left) points. Right: the Hausdorff distance between $\tX$ and the maximal invariant set $[0,1]$ shrinks to $\approx 3\cdot 10^{-2}$ in course of the iteration.}
\label{fig:connecting_1d}
\end{center}
\end{figure}
\end{experiment}

\begin{figure}[H]
\begin{center}
\includegraphics[trim = 1in 3in 1in 3in, clip,width=0.32\textwidth]{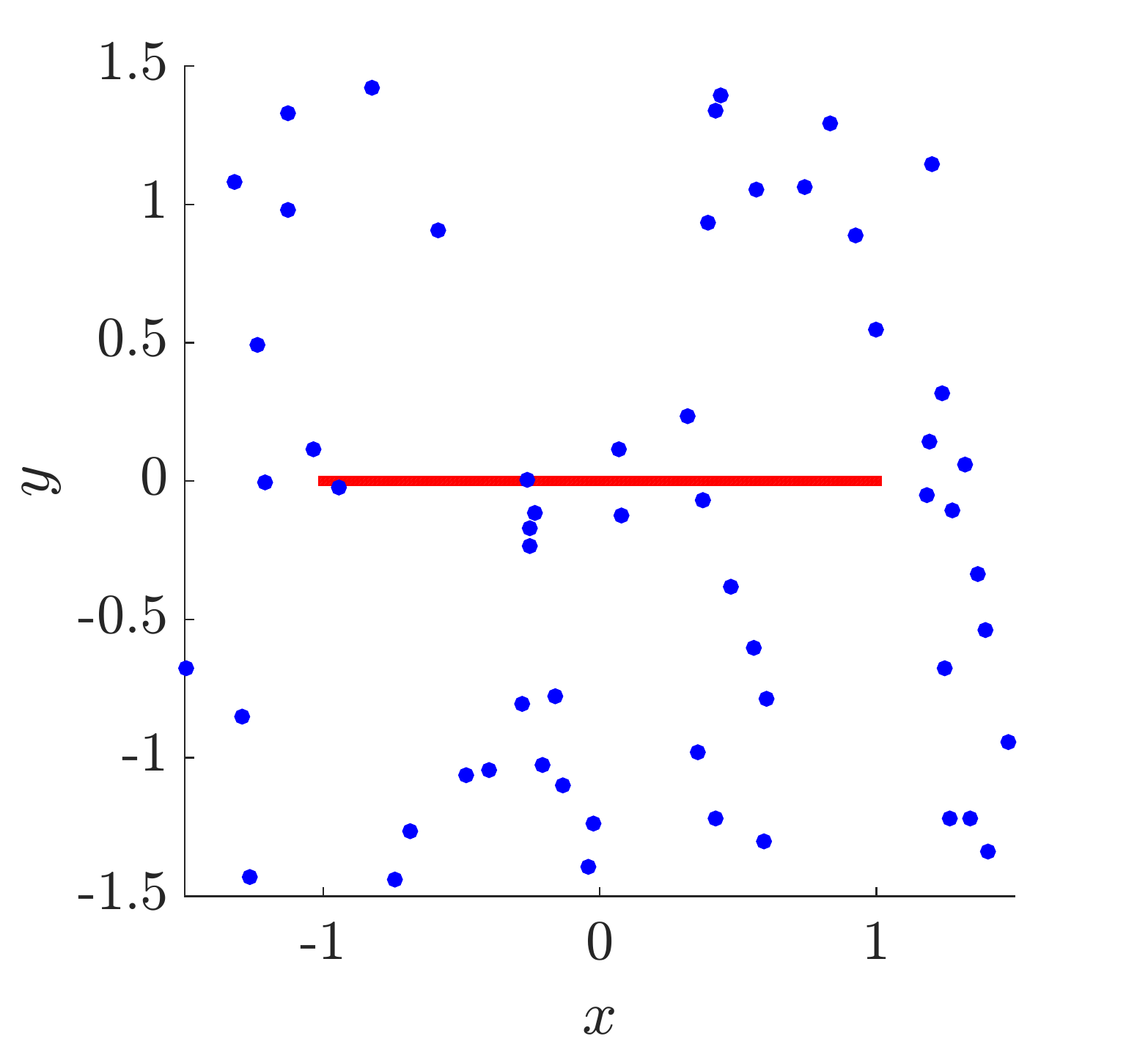}
\includegraphics[trim = 1in 3in 1in 3in, clip,width=0.32\textwidth]{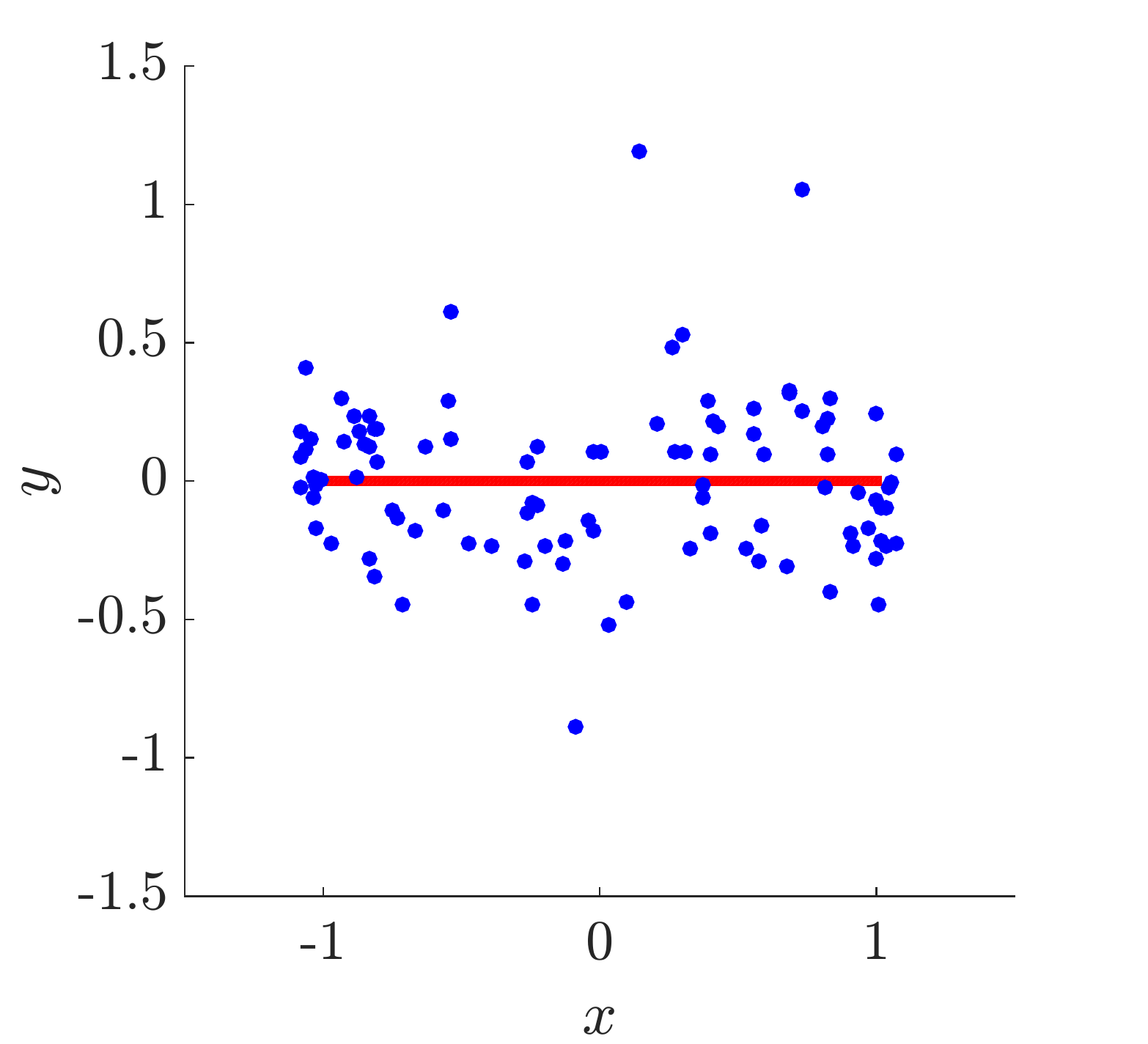}
\includegraphics[trim = 1in 3in 1in 3in, clip,width=0.32\textwidth]{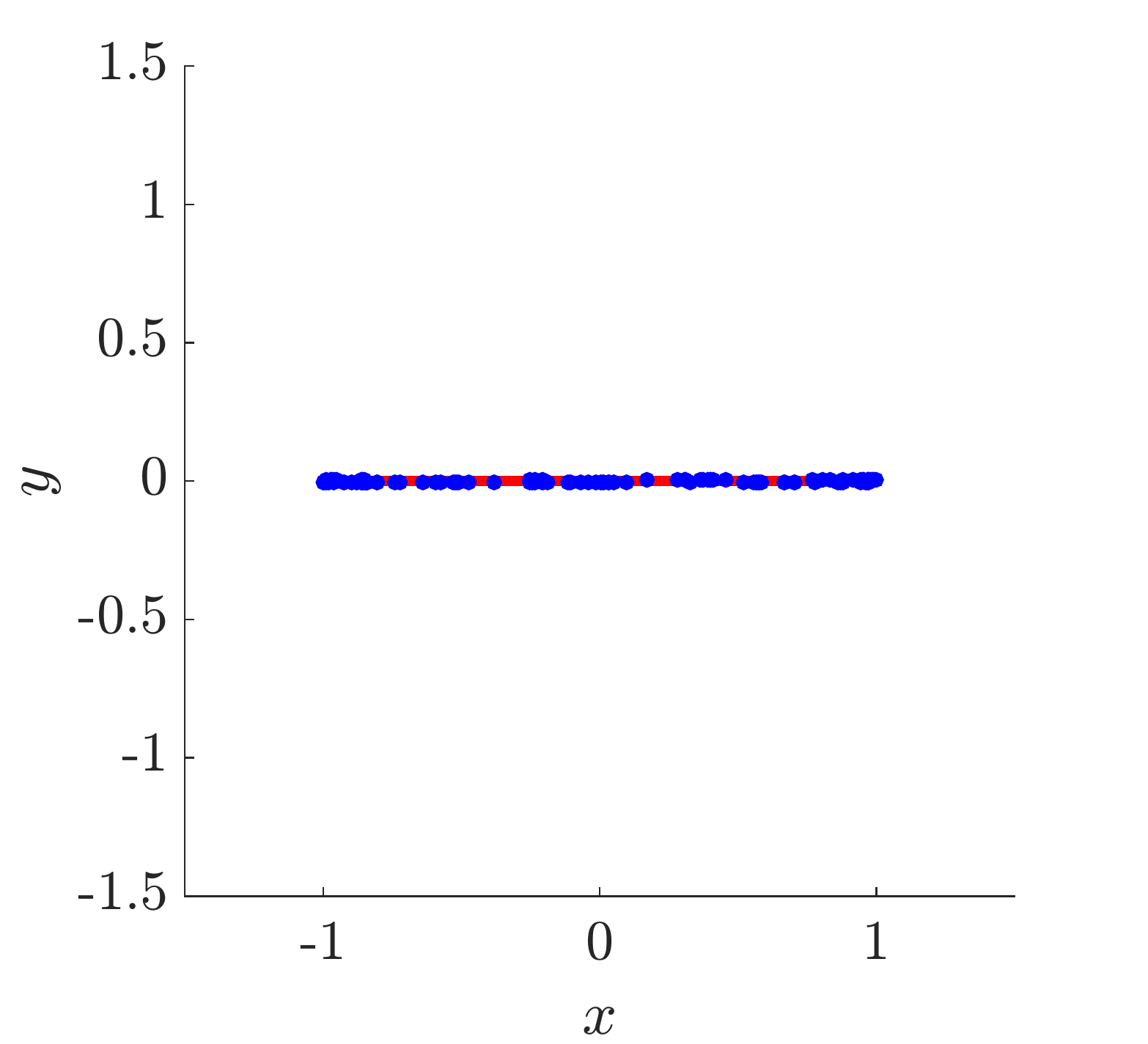}
\caption{Connecting orbit in the plane: the set $\tX$ (blue) initially (left), after 3 (center) and 30 (right) BFGS iterations.}
\label{fig:connecting_2d_X}
\end{center}
\end{figure}

\begin{experiment}[A connecting orbit in 2d]
Similarly, for the map
\[
f(x,y) = (1.5x^3 - 0.5x, 10y)
\]
with fixed points $(-1,0)$ (unstable center), $(0,0)$ (saddle) and $(1,0)$ (unstable center)  the maximal invariant set within any set $Q$ which contains $[-1,1]\times \{0\}$ is $\Inv (Q)=[-1,1]\times \{0\}$.  We start with a set $\tX$ of 100 points which are chosen randomly from $[-2,2]^2$ according to a uniform distribution. Figure~\ref{fig:connecting_2d_X} shows the iterates of $\tX$ in course of the optimization after 3 and 30 BFGS steps. 
\end{experiment}

\begin{experiment}[An \emph{unstable} invariant disk in the plane]\label{exp:disk}
We repeat the experiment with a map for which the maximal invariant set inside a sufficiently large neighborhood is an unstable disk. We consider the vector field
\[
v(x,y) = \begin{bmatrix} -y+ax(x^2+y^2-1)\\x+ay(x^2+y^2-1)\end{bmatrix}
\]
with $a=10$ and define the map $f$ as one Euler step with step size $h=0.1$, i.e.
\[
f(x,y) = (x,y) + hv(x,y).
\]
We start with a set $\tX$ of 1000 points which are chosen randomly from $[-2,2]^2$ according to a uniform distribution. Figure~\ref{fig:circle_2d_X} shows the iterates of $\tX$ in course of the optimization after 3 and 30 BFGS steps.

\begin{figure}[H]
\begin{center}
\includegraphics[trim = 1.2in 3in 1.2in 3in, clip,width=0.4\textwidth]{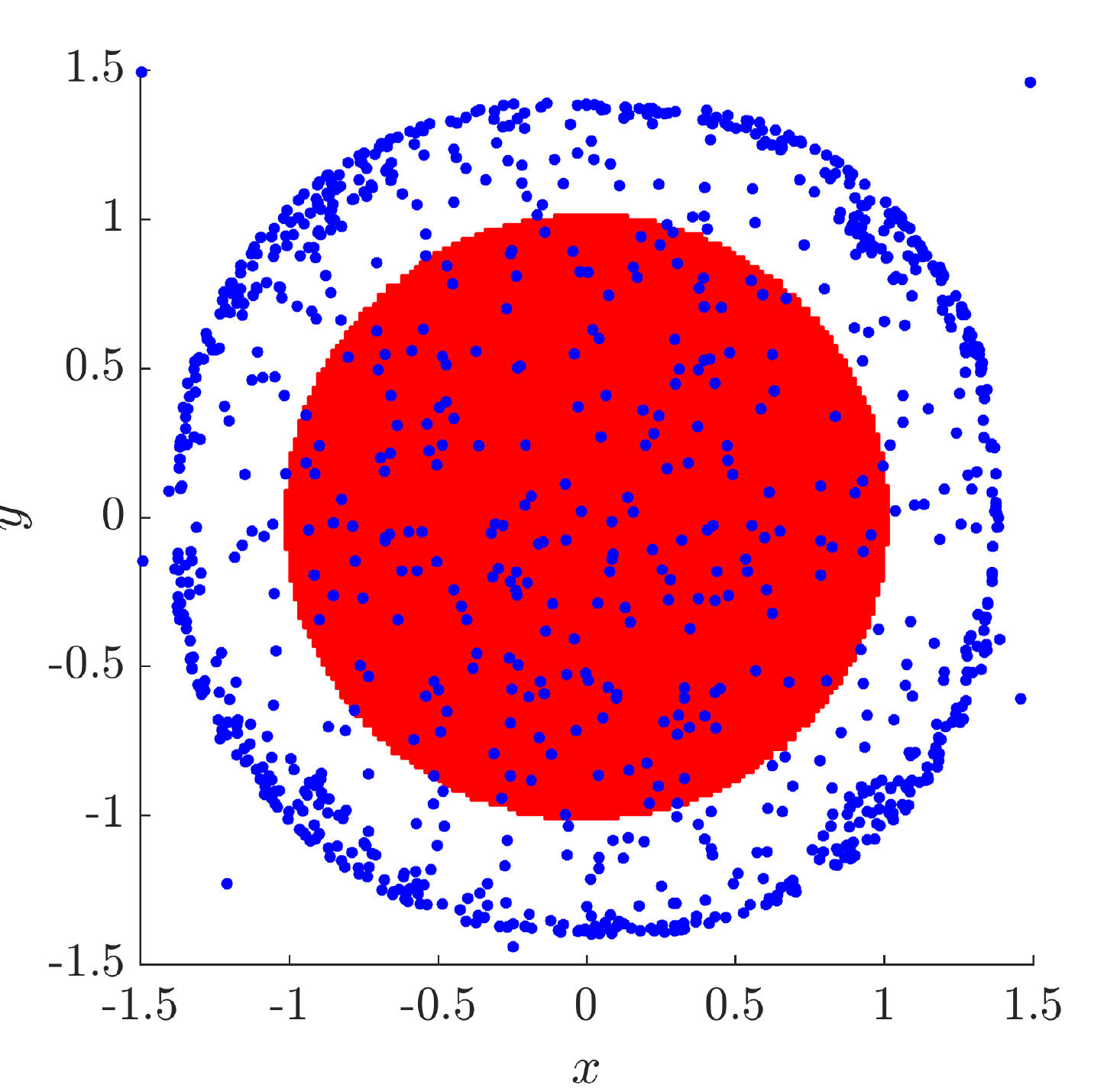}
\includegraphics[trim = 1.2in 3in 1.2in 3in, clip,width=0.4\textwidth]{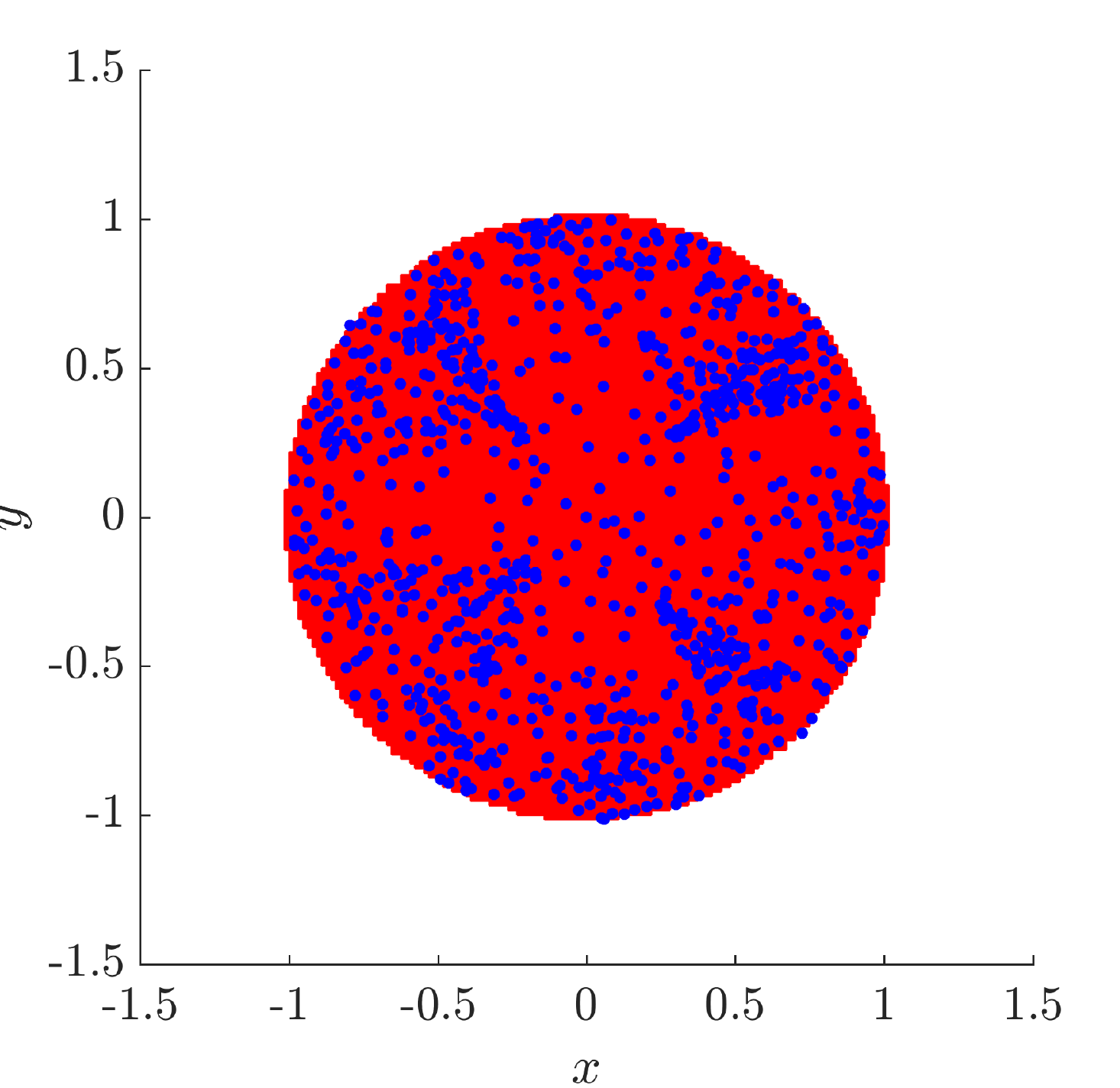}
\caption{Unstable invariant disk (red) in the plane: the set $\tX$ (blue) after 3 (left) and 30 (right) BFGS iterations.}
\label{fig:circle_2d_X}
\end{center}
\end{figure}

Clearly, the objective function $\hat E$ will typically possess many local minimia and the result of the minimization will strongly depend on the initialization of $\tX$.  This is exemplified in Figure~\ref{fig:circle_2d_init}, where the results of the BFGS after 500 iterations is shown for different initializations of $\tX$.  This is one motivation for the construction proposed in Section~\ref{sec:addpot}.

\begin{figure}[H]
\begin{center}
\includegraphics[trim = 1.2in 3in 1.2in 3in, clip,width=0.32\textwidth]{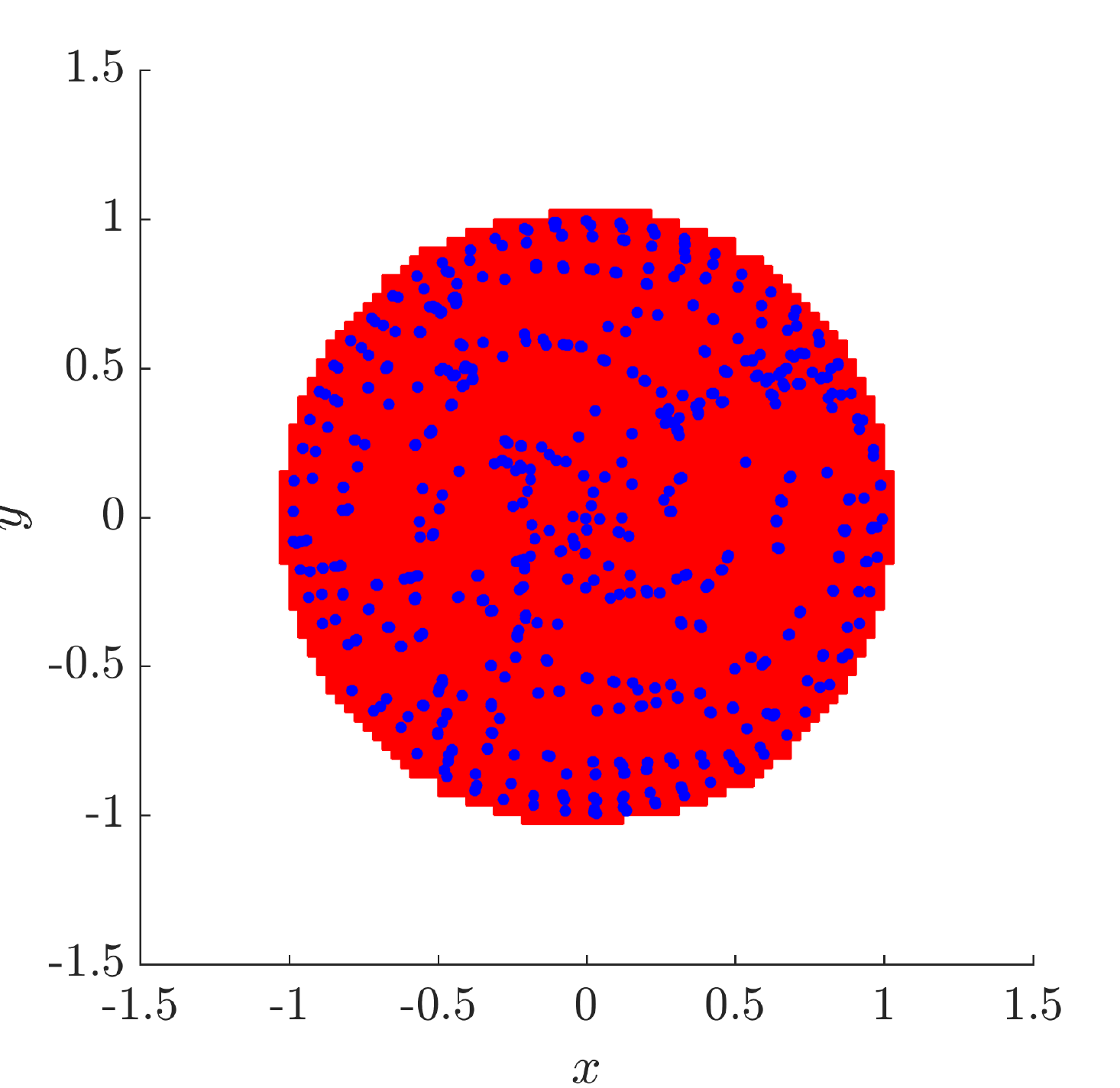}
\includegraphics[trim = 1.2in 3in 1.2in 3in, clip,width=0.32\textwidth]{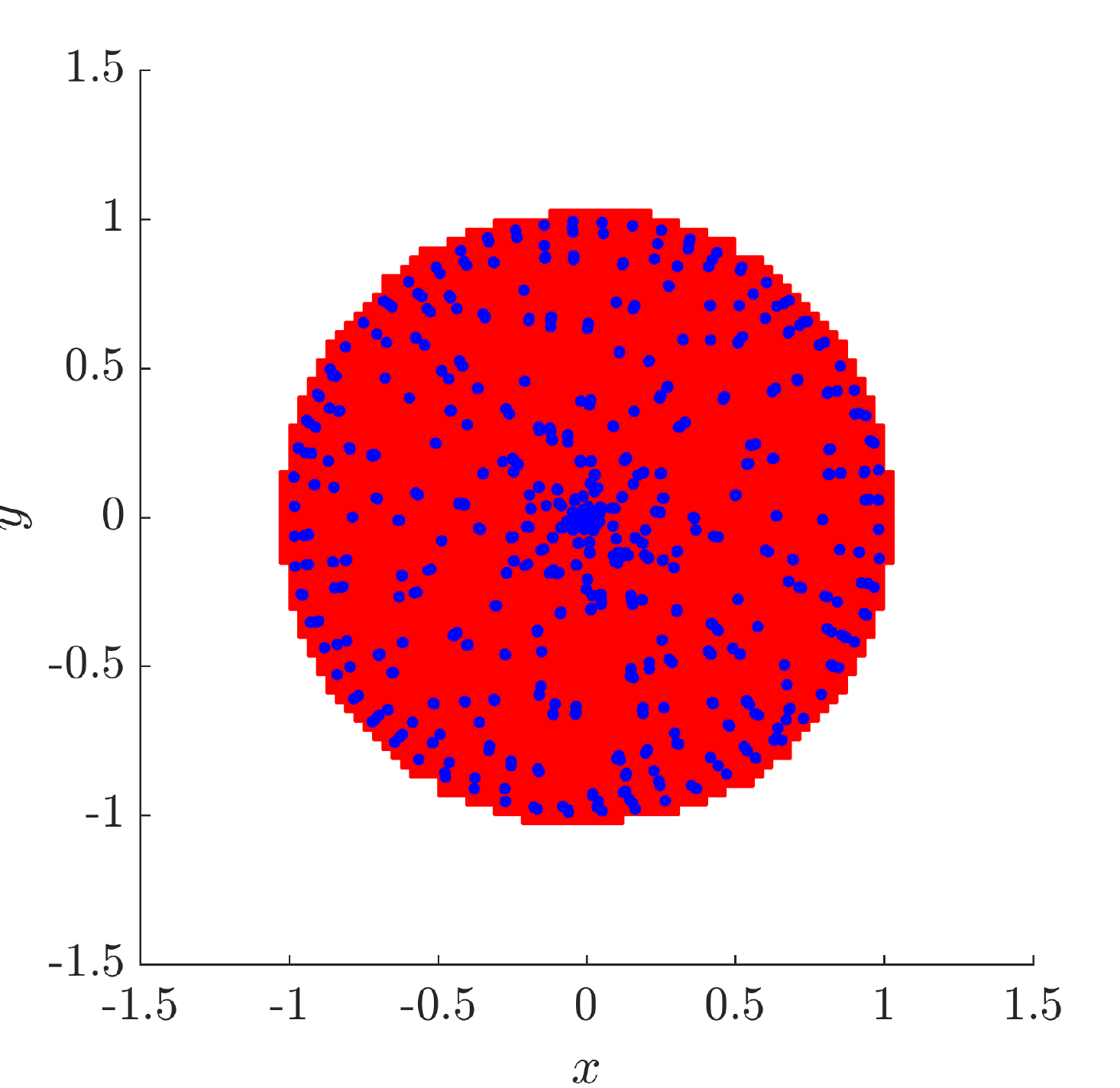}
\includegraphics[trim = 1.2in 3in 1.2in 3in, clip,width=0.32\textwidth]{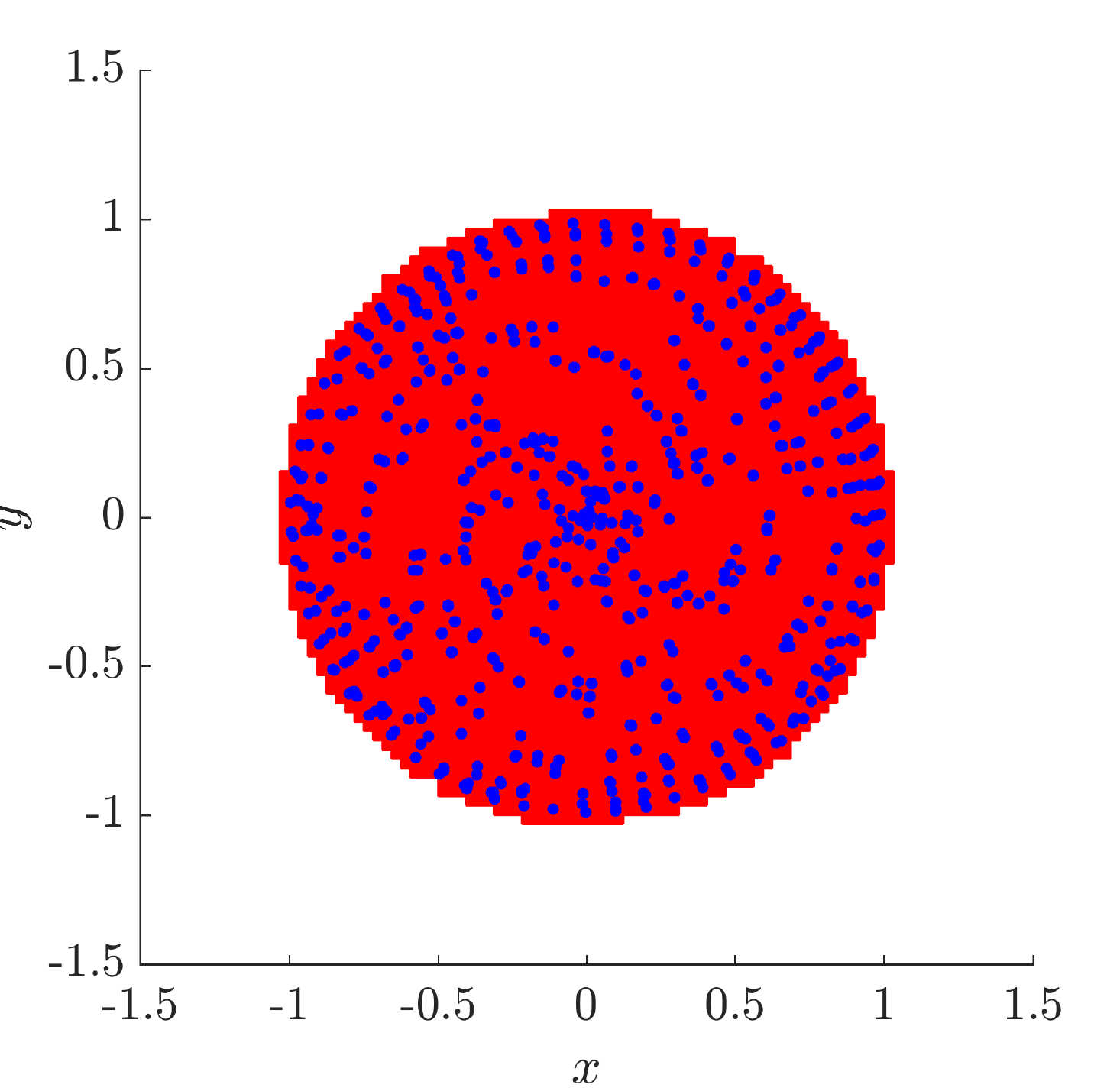}
\caption{Unstable invariant disk (red) in the plane: the 500th iterate of $\tX$ for an initial set $\tX$ of 1000 points chosen from a uniform grid (left), randomly (center) and as pseudo-random points, i.e.\ Halton points \cite{Niederreiter:1992bb}, (right).}
\label{fig:circle_2d_init}
\end{center}
\end{figure}
\end{experiment}

\begin{experiment}[The H\'enon map]\label{exp:henon}
Let us now consider an example with a complicated maximal invariant set as exhibited by the (scaled) H\'enon map
\[
f(x,y) = (1-ax^2+y/3, 3bx),
\]
with parameter values $a=1.3$ and $b=0.3$.  Figure~\ref{fig:henon-mis} shows the attractor (left) as well as a covering of the maximal invariant set (right) as computed by \textsf{GAIO\footnote{\href{https://github.com/gaioguy/GAIO}{\texttt{https://github.com/gaioguy/GAIO}}}} \cite{DeFrJu01a}.  In addition to the attractor, the maximal invariant set contains a saddle fixed point near $(-1.2,-1.2)$ and the piece of its unstable manifold which connects to the attractor.
\begin{figure}[H]
\begin{center}
\includegraphics[trim = 1in 3in 1in 3in, clip,width=0.455\textwidth]{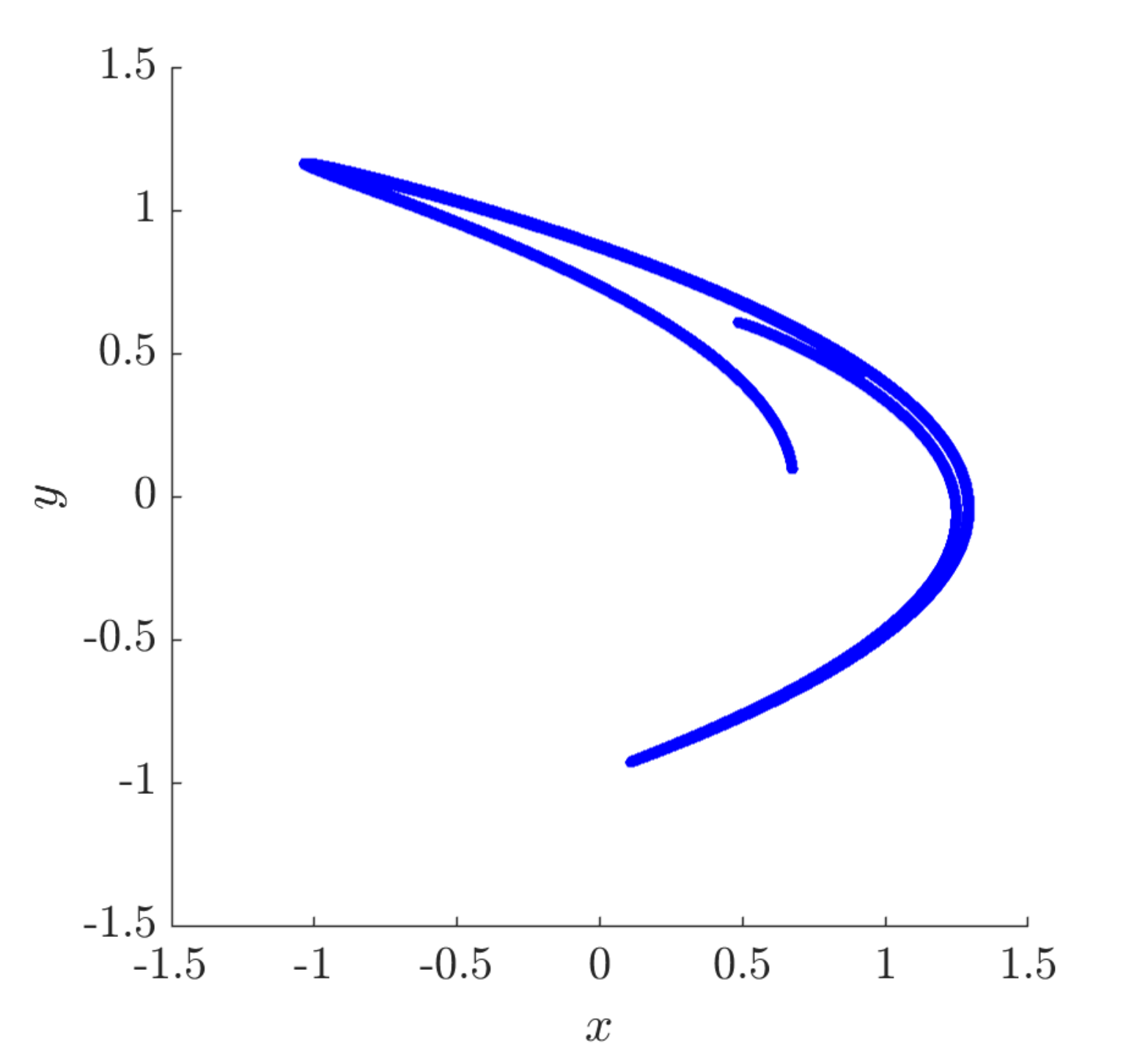}
\includegraphics[trim = 1in 3in 1in 3in, clip,width=0.45\textwidth]{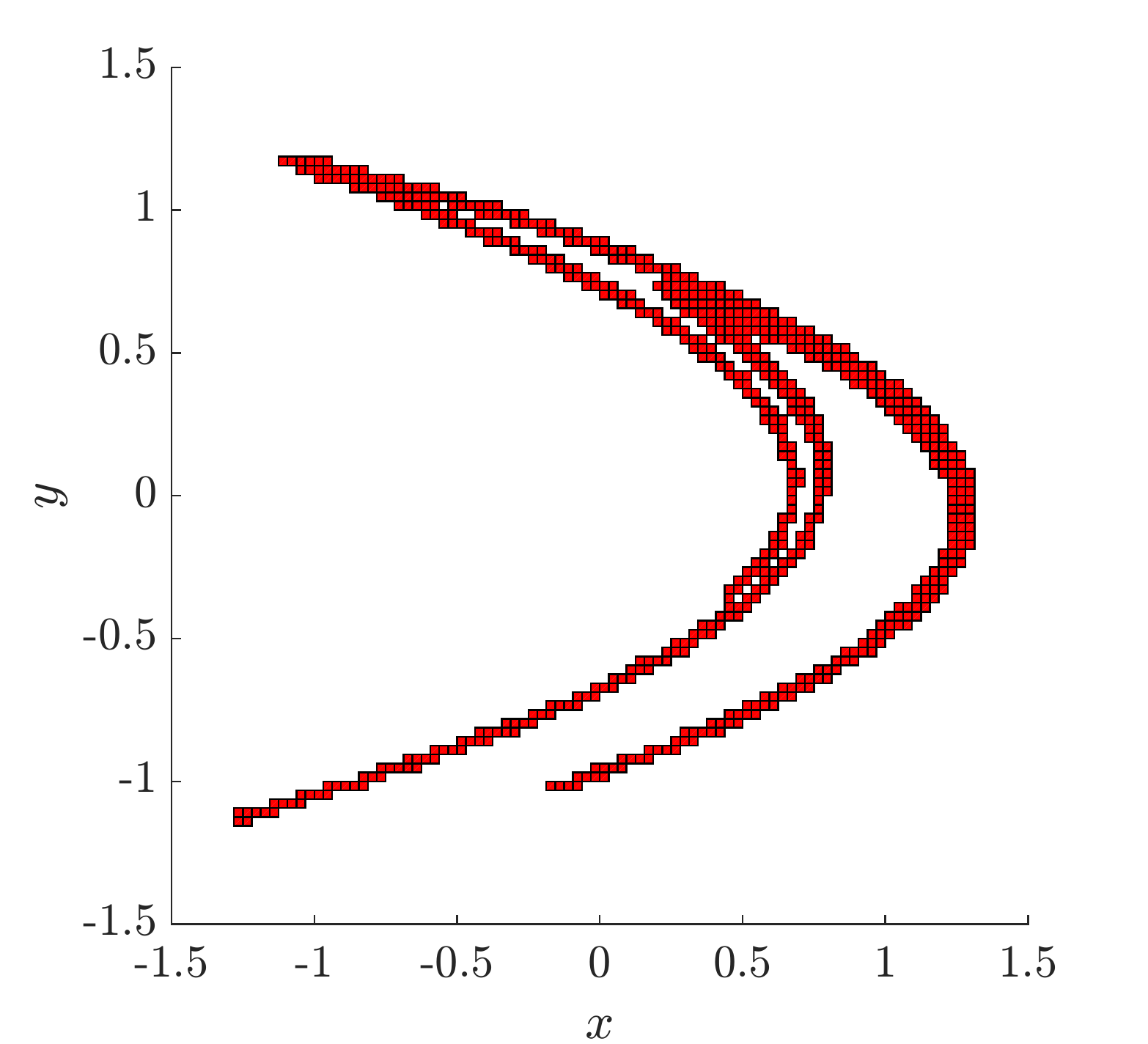}
\caption{H\'enon map: Attractor (left) and box covering of the maximal invariant set as computed by \textsf{GAIO} (right).}
\label{fig:henon-mis}
\end{center}
\end{figure}

We initialize the optimization of $\hat E$ with a set $\tX$ of $1000$ points which have been chosen randomly from the square $[-2,2]^2$ according to a uniform distribution.  

\begin{figure}[H]
\begin{center}
\includegraphics[trim = 1in 3in 1in 3in, clip,width=0.45\textwidth]{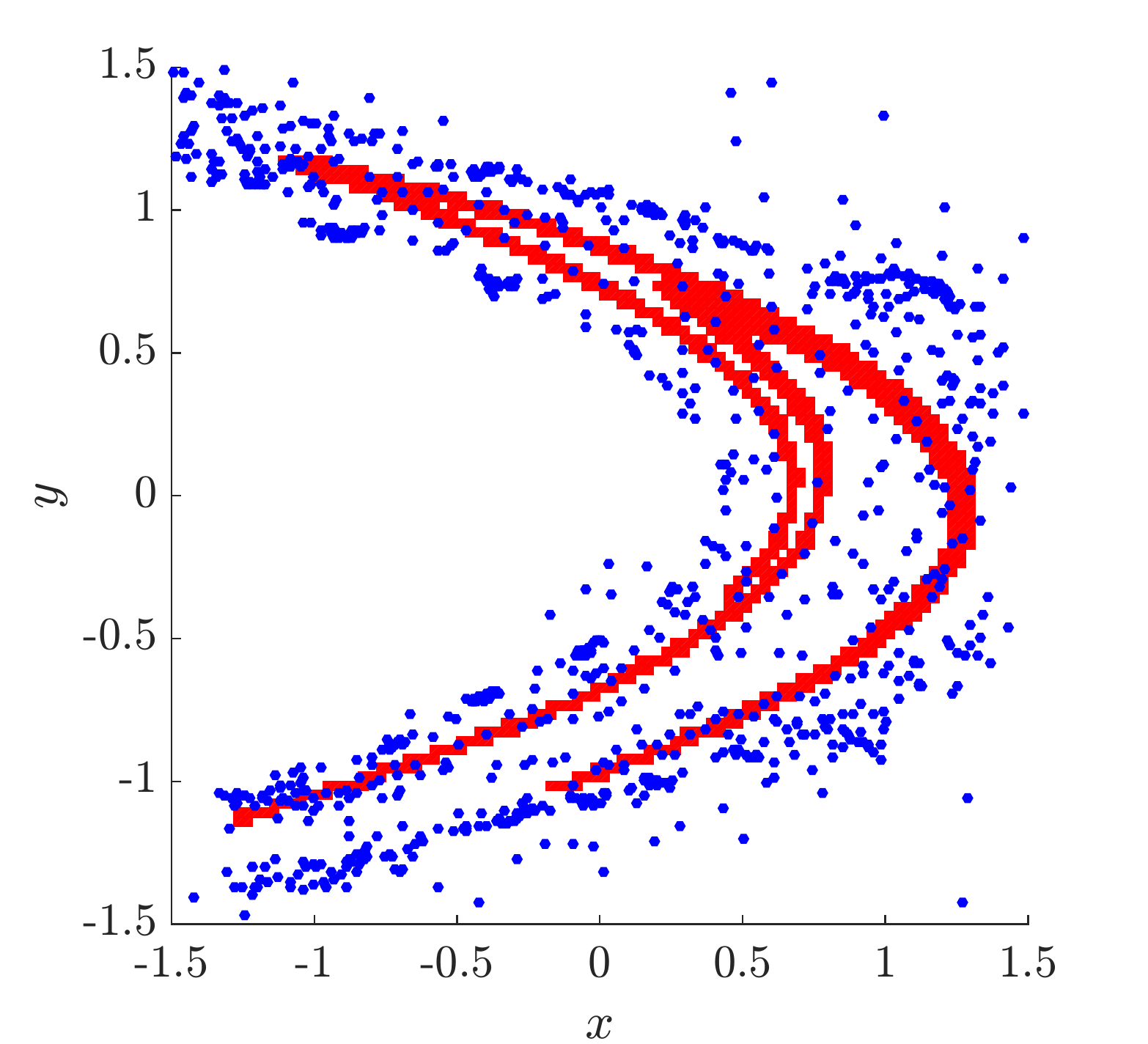}
\includegraphics[trim = 1in 3in 1in 3in, clip,width=0.45\textwidth]{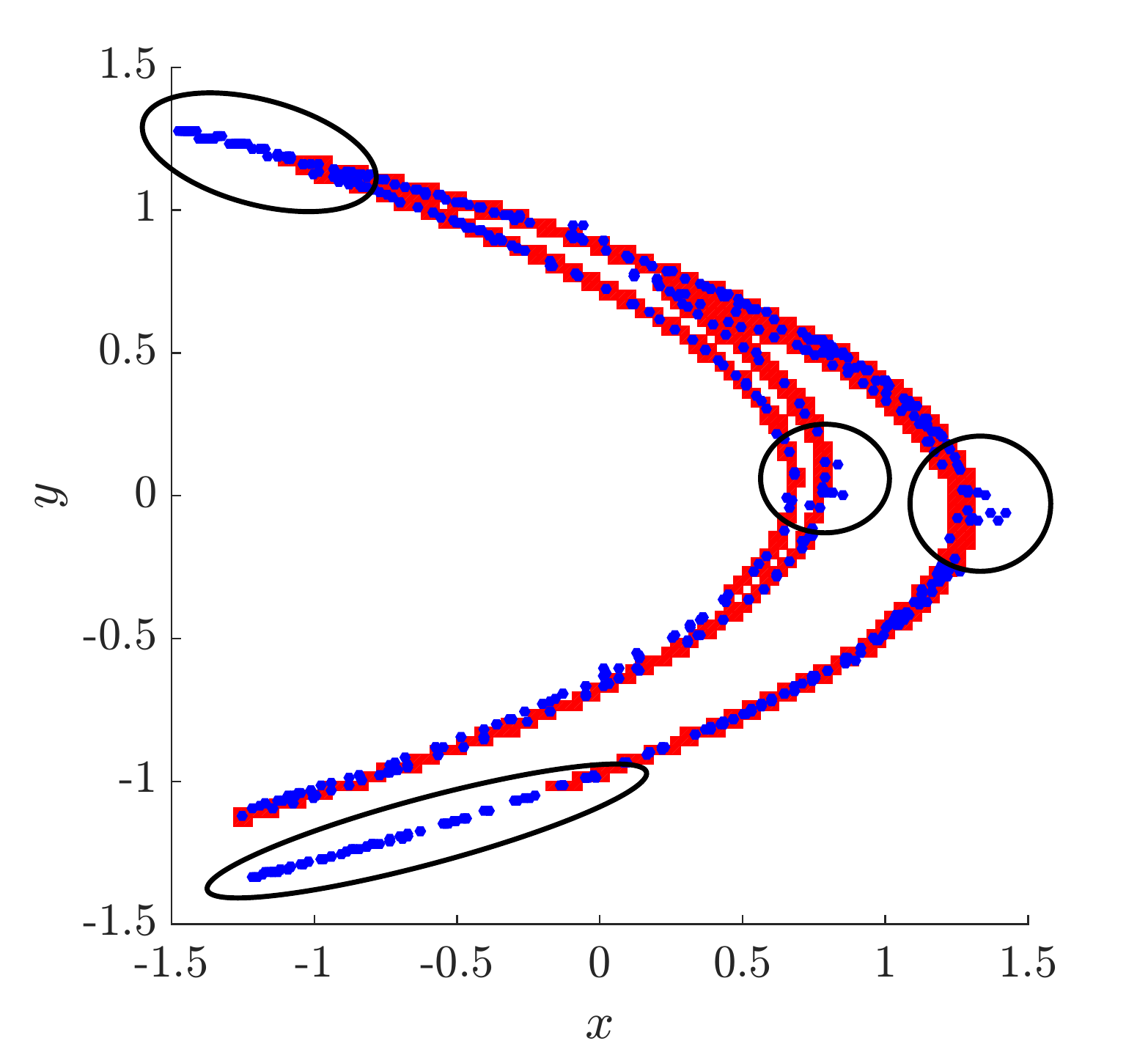}
\caption{Maximal invariant set in the H\'enon map: iterates of an initial set of 1000 randomly chosen points after 20 (left) and 200 (right) steps of the BFGS scheme.}
\label{fig:henon}
\end{center}
\end{figure} 
\end{experiment}

Figure~\ref{fig:henon} shows the iterates of $\tX$ after $20$ and $200$ steps of the optimizer.  Clearly, there appear to be regions (sourrounded by black lines) where points are converging extremely slowly. We conjecture that this is due to (near-)tangencies between stable and unstable manifolds, i.e.\ (near-)nonhyperbolic behaviour.  Note that this phenomenon does not show up near $(-1.25,-1.1)$ where the maximal invariant set is bounded by a saddle fixed point.

\begin{experiment}[A chaotic saddle in 3d]
We finally consider an example in $\R^3$ exhibiting a maximal invariant set with complicated dynamics which is unstable in both time directions.  The map is
\[
f(x,y,z) = (y, z, a+bx+cy-z^2)
\]
with $a = 1.4, b = 0.1, c = 0.3$ which is constructed in analogy to the H\'enon map.  The parameter values have experimentally been chosen such that the maximal invariant set is a saddle.  Note in particular that this set cannot be computed/observed by mere simulation in forward or backward time since the set is unstable in both time directions (this is the unicorn we are alluding to in the title). Figure shows a covering of the maximal invariant set within the cube $[-2,2]^3$ computed by \textsf{GAIO}.
\begin{figure}[H]
\begin{center}
\includegraphics[trim = 1in 3in 1in 3in, clip,width=0.5\textwidth]{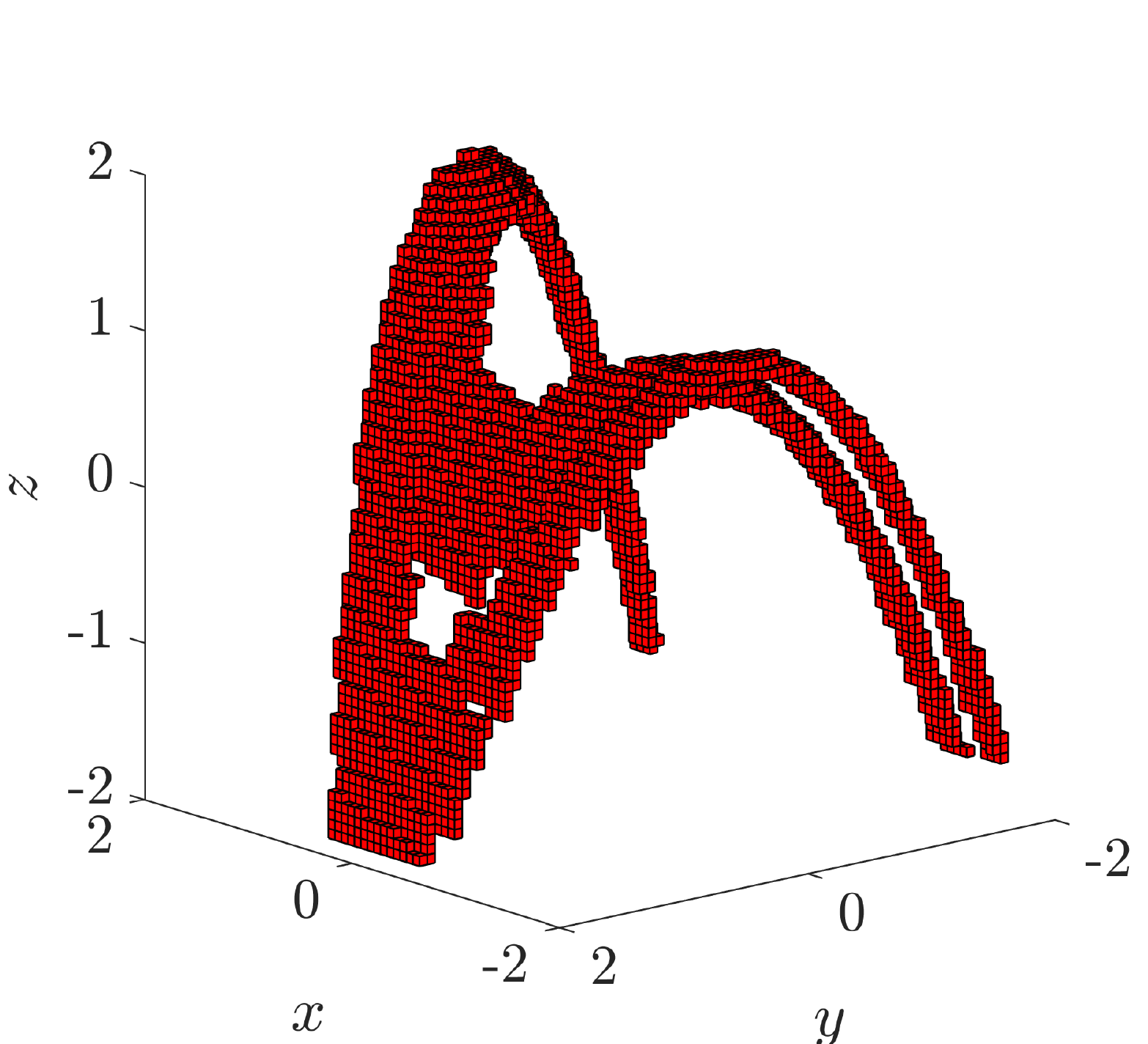}
\caption{3d chaotic saddle: box covering of the maximal invariant set.}
\label{fig:henon3d-mis}
\end{center}
\end{figure}
We initialize the optimization of $\hat E$ with a set $\tX$ of $1000$ points which have been chosen randomly from the cube $[-2,2]^3$ according to a uniform distribution. Figure~\ref{fig:henon3d-iter} shows the iterates of $\tX$ after $20$ and $200$ steps of the optimizer.  Again, we observe slow convergence in certain regions like in the 2d H\'enon example.
\begin{figure}[H]
\begin{center}
\includegraphics[trim = 1in 3in 1in 3in, clip,width=0.49\textwidth]{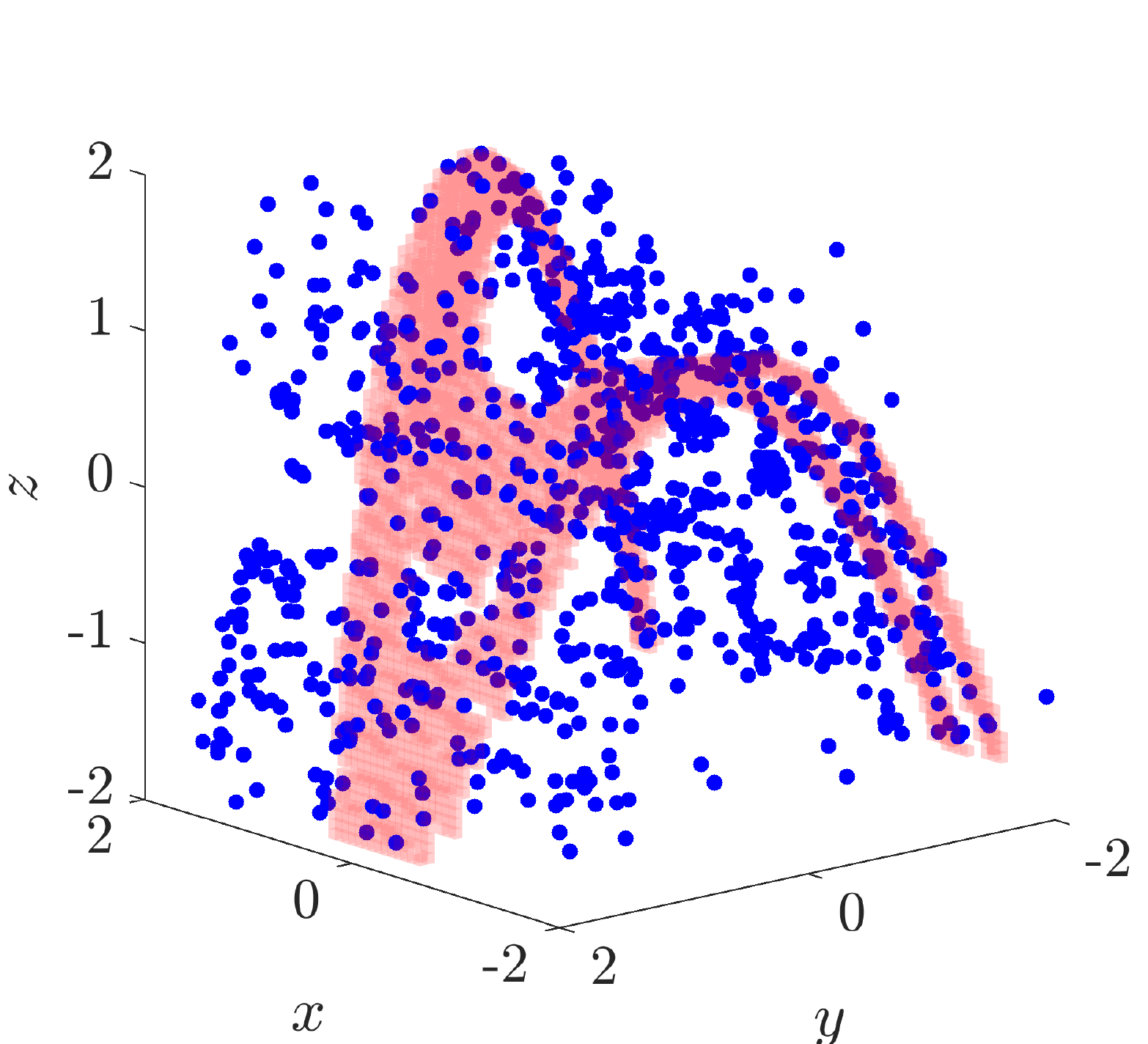}
\includegraphics[trim = 1in 3in 1in 3in, clip,width=0.49\textwidth]{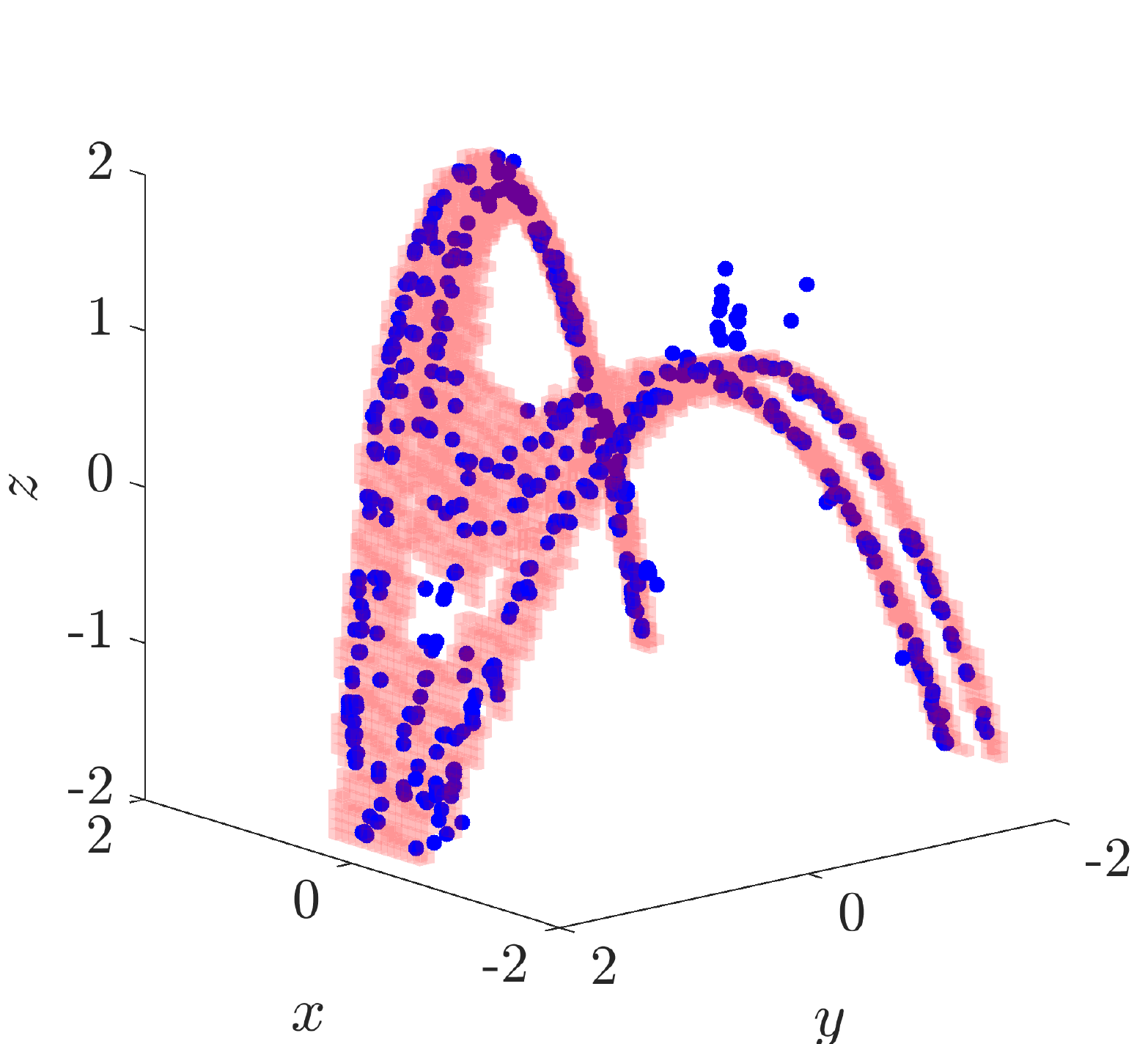}
\caption{3d chaotic saddle: iterates of an initial set of 1000 randomly chosen points after 20 (left) and 200 (right) steps of the BFGS scheme.}
\label{fig:henon3d-iter}
\end{center}
\end{figure} 
\end{experiment}


\section{Additional potentials}
\label{sec:addpot}

While the points in $\tX$ seem to converge towards the maximal invariant set $\Inv(Q)$ in the experiments above, their distribution is typically far from uniform on it. Moreover, one seems to obtain different minimizers depending on the initialization (and also we do not identify vectors which yield the same set, i.e.\ we do not factor by the permutational symmetries of the vector). In fact, in many of the experiments points tend to cluster quite heavily in certain areas and even coincide (cf.\ Fig.~\ref{fig:circle_2d_init}).  In view of our goal to best approximate the maximal invariant set in terms of the Hausdorff distance and to ultimately obtain a unique minimizer, it would be desirable to distribute $\tX$ more uniformly.

As a first step towards this goal, we are going to add a term to the potential $\hat E$ which  strongly penalizes points in $\tX$ from getting too close and which favors them to attain a certain distance $\delta$ to each other.  This can be accomplished by a Lennard Jones potential, cf.\ \cite{Hansen:1990uv},
\[
V_\delta(r) = \left(\frac{\delta}{r}\right)^{2p} - 2\left(\frac{\delta}{r}\right)^p + 1
\]
where the exponent $p\in\N$ controls the ``rigidity'' of the potential and where $r$ is the distance between two points in $\tX$.  In the following experiments, $p=1$ seemed to work best for our purposes.  The proper distance $\delta$ ultimately depends on the dimension of $\Inv(Q)$ and the number $n$ of points in $\tX$ so that we cannot fix the value of $\delta$ a priori and we therefore include $\delta$ as an optimization variable.  One can imagine the Lennard-Jones potential to be a `soft' version of the \emph{hard sphere potential} \cite{Hansen:1990uv} and correspondigly, we here imagine the points in $\tilde X$ to be surrounded by balls of radius $\delta$.

For each point in $\tX$, we are going to restrict the evaluation of $V$ to the $m$ nearest points from $\tX$. The corresponding augmented objective function reads
\begin{align}\label{eq:vp_aug}
J(x_1,\ldots,x_n,\delta) = \hat E(x_1,\ldots,x_n) + \mu\frac{1}{n}\sum_{i=1}^n\frac{1}{m}\sum_{j\in N_m(i)} V_\delta(\|x_i-x_j\|_2),
\end{align}
where $N_m(i)$ is the set of $m$ nearest neighbours of $x_i$ and $\mu>0$ is a weighting parameter.  Larger $\mu$ will favor the points from $\tX$ to attain a lattice structure while smaller $\mu$ favors them to be close to some invariant set.


\subsection{Computational experiments}

\begin{experiment}[On the proper number $m$ of neighbors.] We reconsider Experiment \ref{exp:disk}, choose $\mu=1$ and initialize $\tX$ as a uniform grid of $n=32\times 32=1024$ points within the square $Q=[-2,2]^2$.  We initialize $\delta=\sqrt{m(Q)/(n\pi)}$, i.e.\ such that the sum of the volumes of balls centered at the points in $\tX$ with radius $\delta$ is of the same order as the volume of $Q$.  Figure~\ref{fig:disk_LJ} shows the iterates of $\tX$ after 500 steps of the BFGS scheme for $m=6$ (left) and $m=30$ (right).  The larger number of neighbors yields a much better approximation.
\begin{figure}[H]
\begin{center}
\includegraphics[trim = 1in 3in 1in 3in, clip,width=0.49\textwidth]{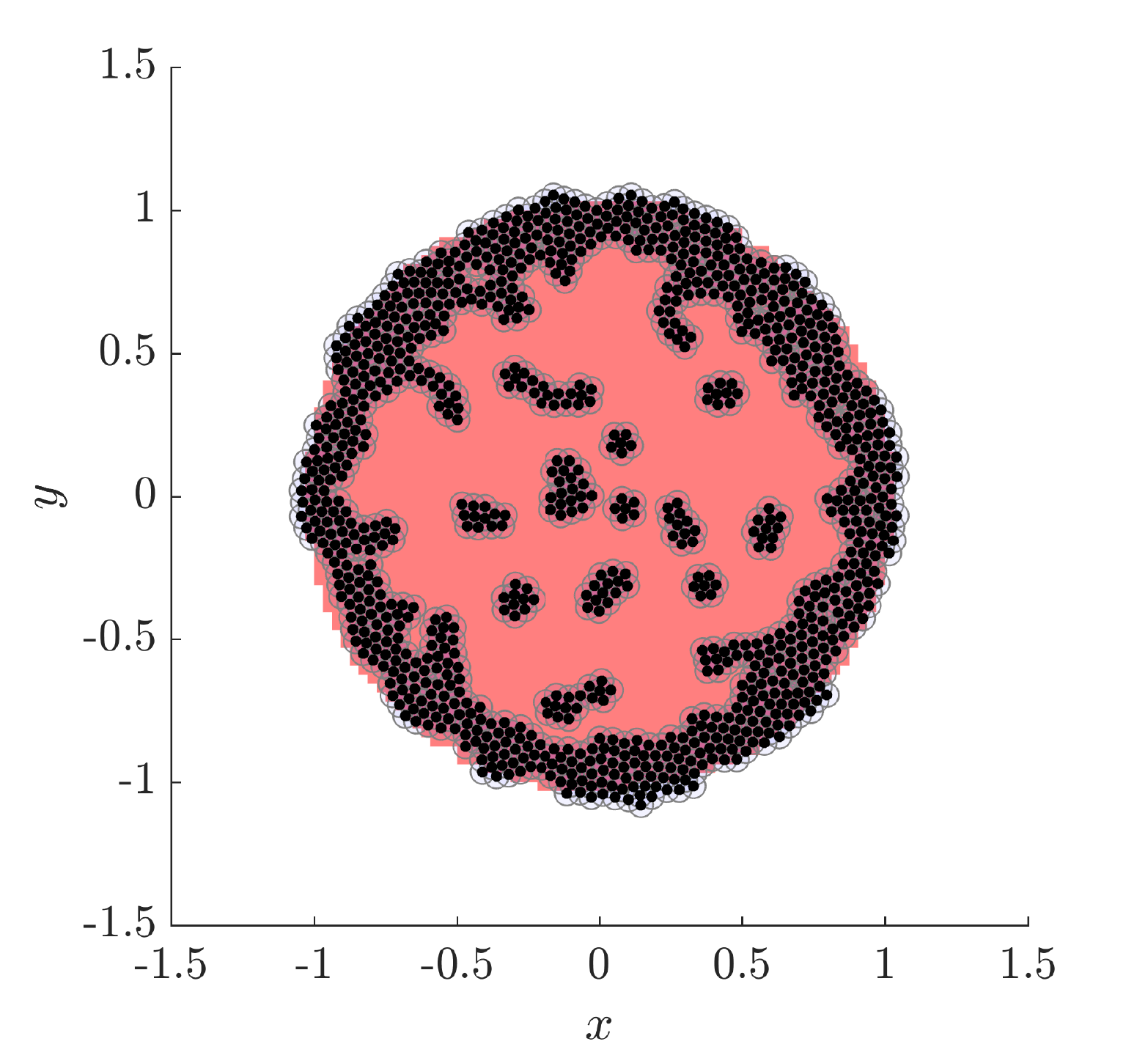}
\includegraphics[trim = 1in 3in 1in 3in, clip,width=0.49\textwidth]{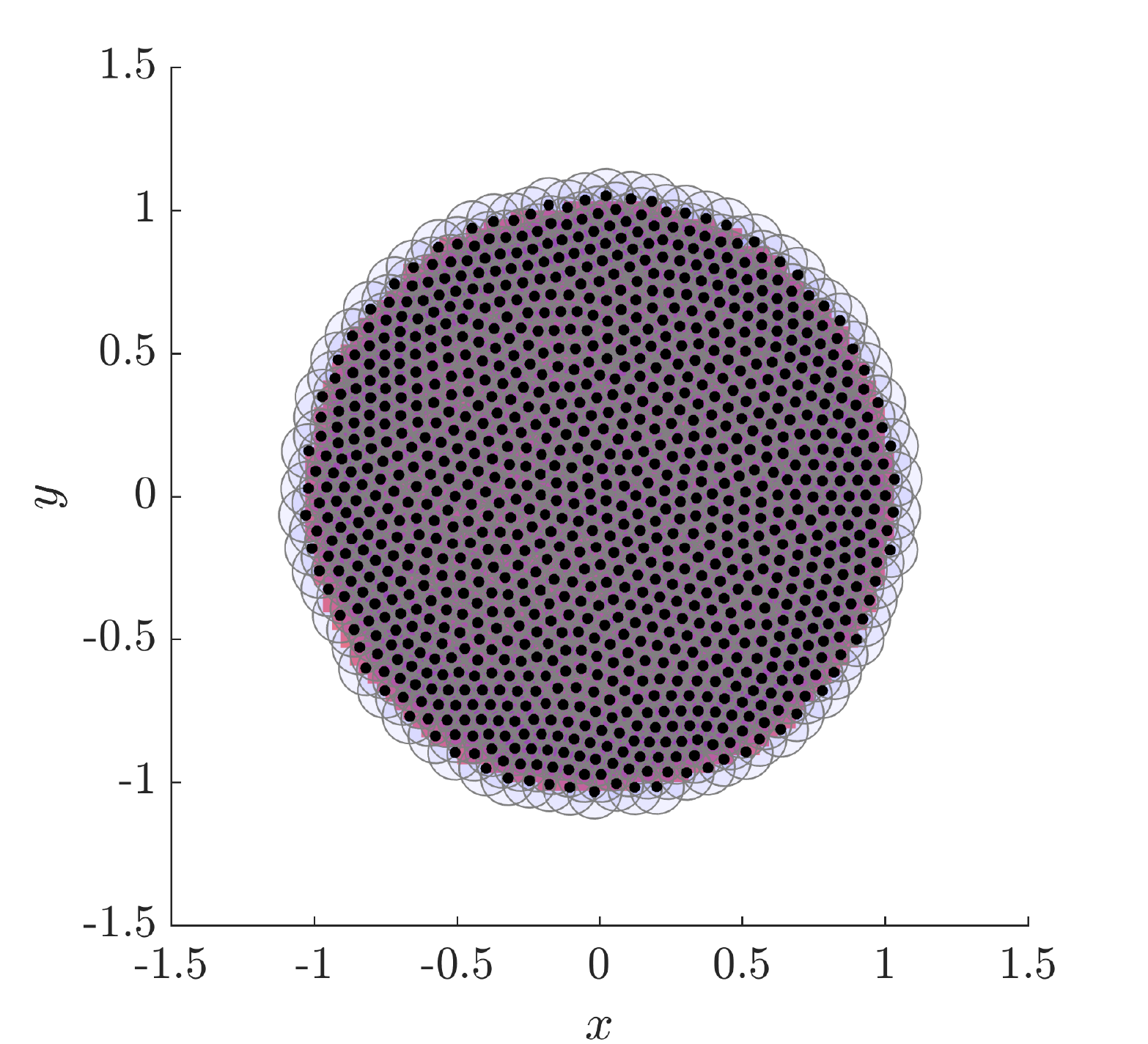}
\caption{Invariant disk, with Lennard-Jones potential: 500th iterate of the initial point cloud for $m=6$ (left) and $m=30$ (right). We show the set $\tX$ (black dots) together with surrounding balls of radius $\delta$ (where $\delta$ results from the optimization).}
\label{fig:disk_LJ}
\end{center}
\end{figure} 

Figure \ref{fig:henon_LJ} shows the results of the same experiment, albeit for the H\'enon map. While a larger number of neigbours tends to yield a more uniform covering of the maximal invariant set here as well, it also tends to hide finer structures (given a fixed number of balls).

\begin{figure}[H]
\begin{center}
\includegraphics[trim = 1in 3in 1in 3in, clip,width=0.45\textwidth]{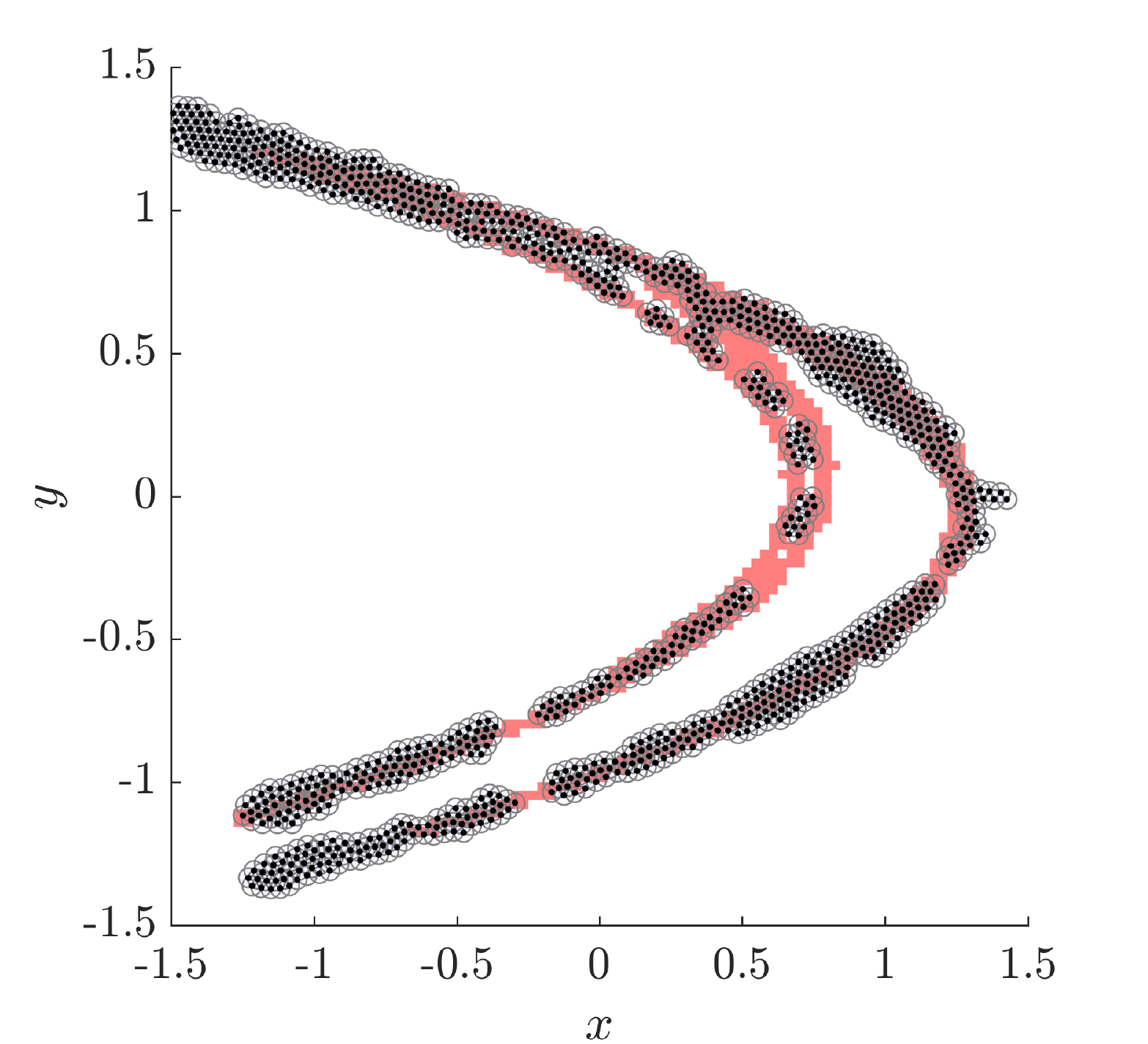}
\includegraphics[trim = 1in 3in 1in 3in, clip,width=0.45\textwidth]{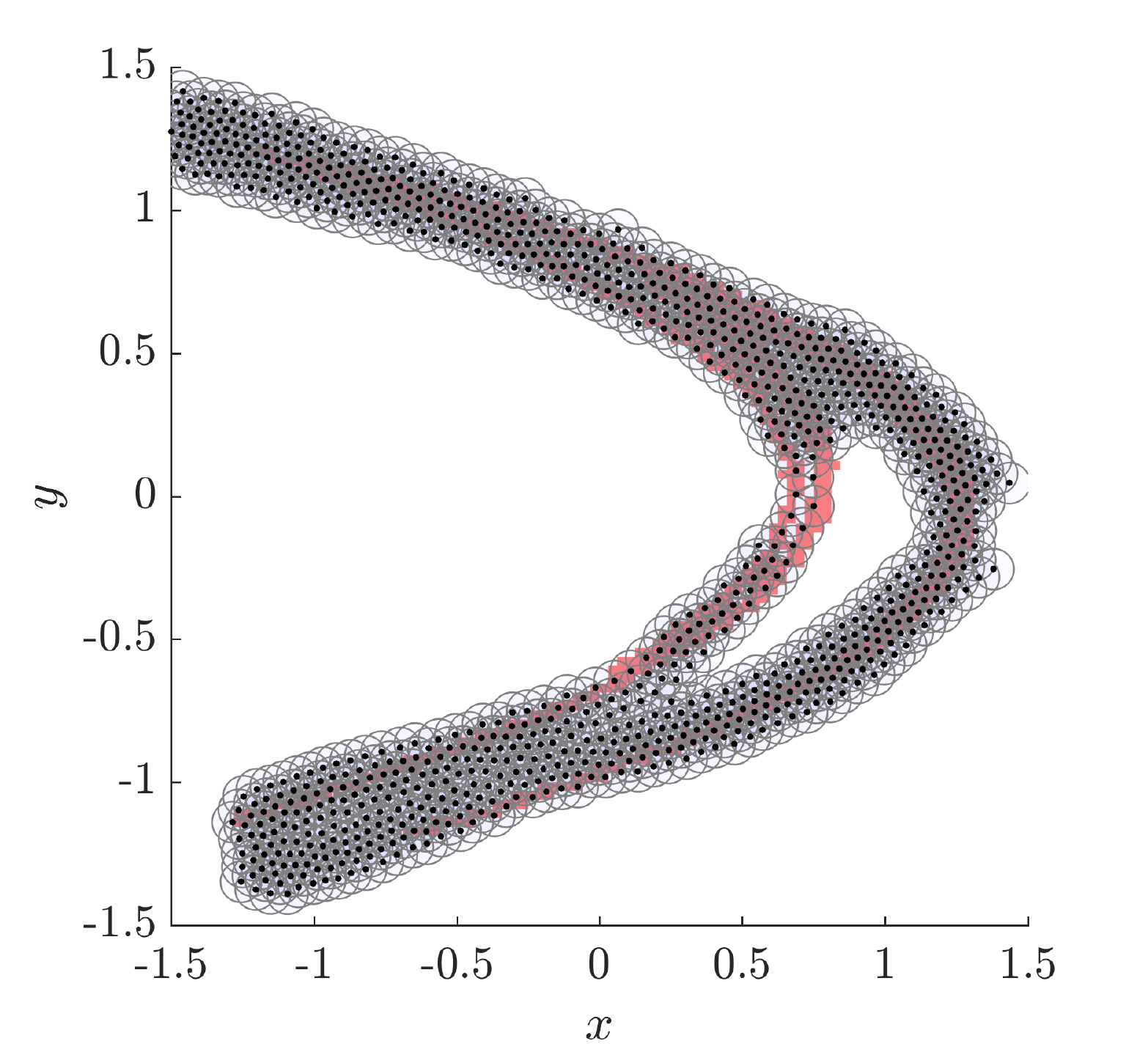}
\caption{H\'enon map, with Lennard-Jones potential: 500th iterate of the initial point cloud for $m=6$ (left) and $m=30$ (right).}
\label{fig:henon_LJ}
\end{center}
\end{figure} 
\end{experiment}

\begin{experiment}[On the choice of $\mu$.] This latter drawback can be alleviated by decreasing the value of $\mu$, i.e.\ decreasing the influence of the Lennard Jones term, as evidenced by repeating the previous experiment on the H\'enon map with $\mu=0.01$ (Fig.~\ref{fig:henon_LJ_mu}).   

\begin{figure}[H]
\begin{center}
\includegraphics[trim = 1in 3in 1in 3in, clip,width=0.5\textwidth]{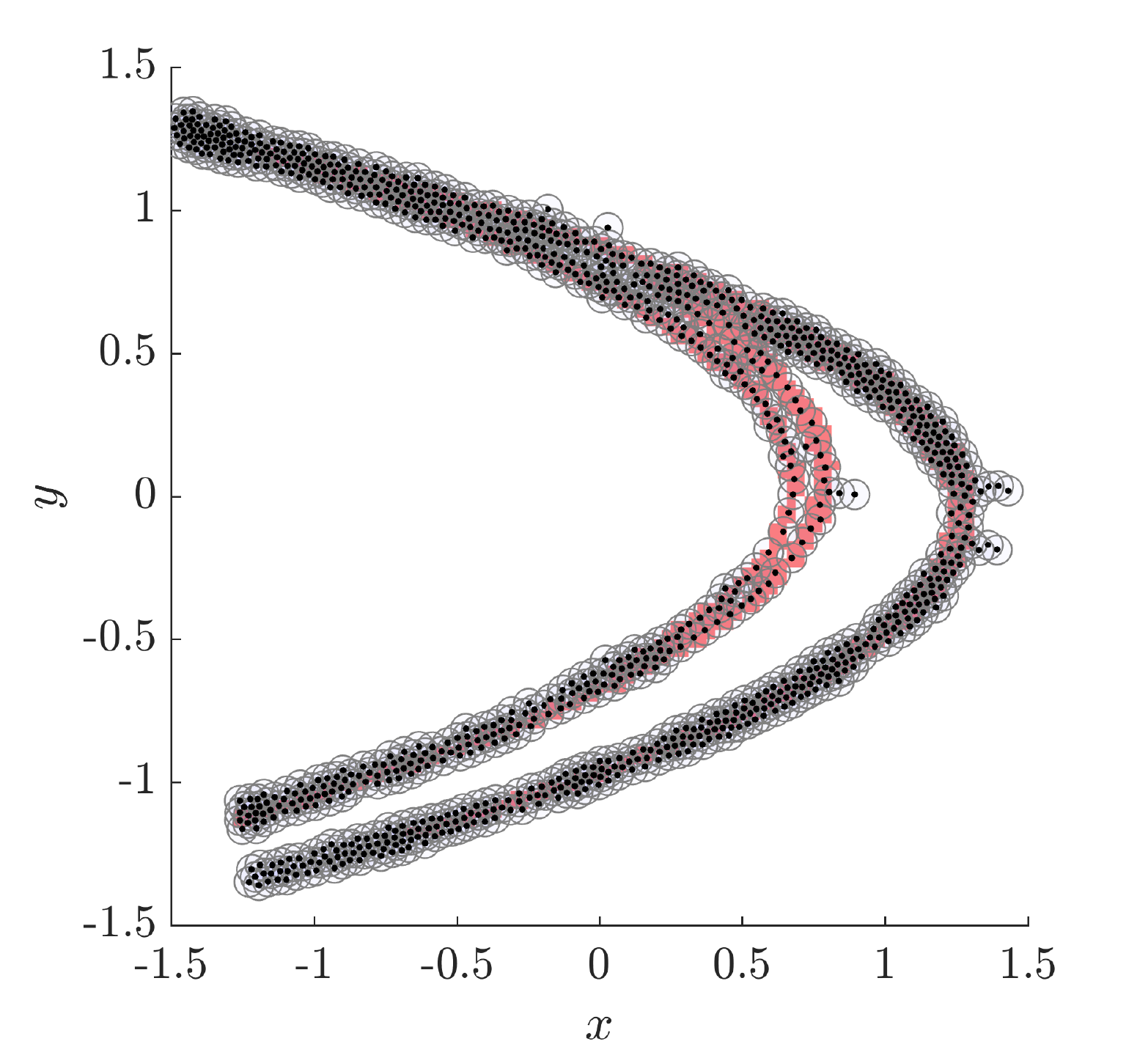}
\caption{H\'enon map, with Lennard-Jones potential: 500th iterate of the inital point cloud, $\mu=0.01$.}
\label{fig:henon_LJ_mu}
\end{center}
\end{figure} 
\end{experiment}

\section{Discussion and future directions} 

Clearly, the experiments in this paper can only be seen as a first step. Of course it would be desirable to gain insight into the general convergence behavior of the scheme, in particular as the number of points goes to infinity and this is currently under investigation.  Further, it would desirable to alleviate the bad convergence behavior in weakly hyperbolic regions.  

While the inclusion of the Lennard-Jones potential seems to point in the right direction, it also raises new issues like the proper number of neighbors and the proper value of the weighting parameter $\mu$. Of course, other potentials might be conceivable as well.  In particular, it might be useful to adapt the `radius parameter' $\delta$ locally, i.e.\ use balls of smaller radius where appropriate.  A multilevel scheme might be useful where one considers balls of several scales at the same time in the spirit of the famous
``cheese theorem'' of E.~Lieb \cite{Lieb:1972gr}.

As mentioned, in principle any metric on the set of compact subsets of $\R^d$ will do.  Our choice of a Hausdorff type distance was motivated by smoothness considerations.   A natural candidate for a different choice would be the Wasserstein or earth mover's distance (where $\tX$ is seen as a sum of atomic measures).  We will explore whether this bears any advantage over the Hausdorff type distance used here (in particular, since the numerical effort for computing the Wasserstein distance is presumably larger than for the Hausdorff type metric).

With the limited memory BFGS scheme, we used a standard quasi-Newton method for the minimization of the objective functional.  Depending on the set-metric employed, other schemes might be more beneficial, both from a theoretical point of view (in order to prove convergence) and also from a numerical efficiency point of view. 

As mentioned, one of the motivations for considering the approach advocated in this paper was to construct an approximation of some invariant set which varies smoothly in the case that the underlying invariant set varies smoothly with some system parameter.  In fact, it is an interesting question whether our approach can be embedded into a path following scheme.  

Another interesting question is how to modify the functional $E$ such that an invariant set of particular type is computed, e.g.\ is it possible to directly compute the chain recurrent set instead of the maximal invariant one.

\section{Acknowledgements}

We thank Gero Friesecke and William Leeb for helpful discussions as well as Daniel Karrasch for careful proofreading and helpful suggestions.  We also gratefully acknowledge support by Institute of Advanced Studies at the Technical University Munich.

\end{document}